\documentclass[11pt]{article}

\usepackage{times}
\usepackage[hyperindex=true,pageanchor=true,hyperfigures=true,backref=false]{hyperref} 
\usepackage{amsmath,amsthm}
\usepackage{amssymb}
\usepackage{dsfont}

\usepackage[latin1]{inputenc}
\usepackage{textcomp,url}
\usepackage{tikz}

\usepackage{stmaryrd}

\DeclareSymbolFont{wideparensymbol}{OMX}{yhex}{m}{n}
\DeclareMathAccent{\wideparen}{\mathord}{wideparensymbol}{"F3}

\newcommand\firstcircle{(6,-4) circle (1.5cm)}
\newcommand\firstellipse{(4.5,-4) ellipse (5cm and 2cm)}
\newcommand\secondellipse{(7.5,-4cm) ellipse (5cm and 2cm)}
\newcommand\thirdellipse{(6,-4cm) ellipse (3.6cm and 1.8cm)}
\newcommand\fourthellipse{(6,-4cm) ellipse (6cm and 2.1cm)}

\usepackage{amsmath}
\usepackage{amssymb}
\DeclareSymbolFontAlphabet{\Bbb}{AMSb}

\newtheorem{theorem}{Theorem}[section]
\newtheorem{lemma}[theorem]{Lemma}

\newtheorem{definition}[theorem]{Definition}
\newtheorem{assumption}[theorem]{Assumption}

\theoremstyle{definition}
\newtheorem{example}[theorem]{Example}

\theoremstyle{remark}

\newtheorem{remark}[theorem]{Remark}

\newlength{\fixboxwidth}
\newcommand{\ca}[1]{{\cal #1}}

\newcommand{\R}{\mathbb{R}}

\renewcommand{\a}{\alpha}

\newcommand{\g}{\gamma}

\renewcommand{\d}{\delta}

\newcommand{\e}{\varepsilon}

\newcommand{\lb}{\lambda}
\newcommand{\x}{\xi}

\newcommand{\s}{\sigma}

\renewcommand{\t}{\tau}

\renewcommand{\P}{\mathrm{P}}

\newcommand{\Om}{\Omega}

\newcommand{\snorm}[1]{\Vert #1 \Vert}

\newcommand{\inorm}[1]{\Vert #1 \Vert_\infty}

\newcommand{\RP}[2]{{{\cal R}_{#1,\P}(#2)}}
\newcommand{\RPB}[1]{{{\cal R}_{#1,\P}^{*}}}

\makeatletter
\g@addto@macro{\example}{\small}
\makeatother

\title{A Bernstein-type Inequality for Some Mixing Processes and Dynamical Systems with an Application to Learning}

\author{H. Hang and I. Steinwart \\ 
\small{Institute for Stochastics and Applications} \\
\small{University of Stuttgart} \\
\small{D-70569 Stuttgart} \\
\small{\texttt{\{hanghn,ingo.steinwart\}@mathematik.uni-stuttgart.de}}}

\begin{document}

\maketitle

\begin{abstract}
We establish a Bernstein-type inequality 
for a class of stochastic processes that include 
the classical geometrically $\phi$-mixing processes,
Rio's generalization of these processes, as well as many time-discrete dynamical systems.
Modulo a logarithmic factor and some constants, our Bernstein-type inequality 
coincides with the classical Bernstein inequality for i.i.d.~data.
We further use this new Bernstein-type inequality to derive
an oracle inequality for generic regularized 
empirical risk minimization algorithms and data generated by such processes.
Applying this  oracle inequality
to 
support vector machines   using the Gaussian  kernels 
for both least squares and quantile regression,
it turns out that the resulting learning rates match, up to 
some arbitrarily small extra term in the exponent, 
the optimal rates for i.i.d.~processes.
\end{abstract}

\section{Introduction} \label{introduction}


Concentration inequalities such as Hoeffding's inequality, Bernstein's inequality, McDiarmid's inequality,
and Talagrand's inequality play an important role in many areas of probability.
For example, the analysis of various methods from non-parametric statistics and machine learning 
crucially depend on these inequalities, see e.g.~\cite{DeGyLu96,DeLu01,GyKoKrWa02,StCh08a}.
Here,  stronger results can typically be achieved by Bernstein's inequality and/or Talagrand's inequality,
since these inequalities allow for localization due to their specific dependence on the variance.
In particular, most derivations of minimax optimal learning rates are based on one of these inequalities.

%

The concentration inequalities mentioned above all assume the data to be generated by an i.i.d.~process.
Unfortunately, however, this
assumption is often violated in several important areas of 
applications including 
financial prediction, signal processing, system observation and diagnosis, text and speech recognition, and  
time series forecasting.
For this and other reasons there has been some effort 
to establish concentration inequalities for non-i.i.d.~processes, too.
For example,  generalizations of Bernstein's inequality 
to   $\alpha$-mixing  and $\phi$-mixing processes  have been found  \cite{Bosq93a,MoMa96a,MePeRi09a} and 
\cite{Samson00a}, respectively. Among many other applications, the Bernstein-type inequality
established in \cite{Bosq93a} was used in \cite{Zhang04a} to obtain  convergence rates for sieve estimates 
from $\alpha$-mixing strictly stationary processes in the special case of neural networks.
Furthermore, \cite{HaSt14a} applied the Bernstein-type inequality in \cite{MoMa96a} to derive
an oracle inequality for generic regularized empirical risk minimization algorithms 
learning from stationary $\alpha$-mixing processes. 
Moreover, by employing the Bernstein-type inequality in \cite{MePeRi09a}, 
\cite{Belomestny11a} derived almost sure uniform rates of convergence for the estimated L\'{e}vy density 
both in mixed-frequency and low-frequency setups and proved that these rates are optimal in the minimax sense.
Finally, in the particular case of the least square loss, 
\cite{AlWi12a} obtained the optimal learning rate for $\phi$-mixing processes by applying
the Bernstein-type inequality established in \cite{Samson00a}.

However, there exist many dynamical systems 
such as the uniformly expanding maps given in \cite[p.~41]{DeDoLaRLoPr07a}
that are not $\a$-mixing.
To deal with such non-mixing processes
Rio \cite{Rio96a} introduced so-called $\tilde \phi$-mixing coefficients, 
which extend the classical $\phi$-mixing coefficients.
For dynamical systems 
with exponentially decreasing, \emph{modified}   $\tilde \phi$-coefficients,
\cite{Wintenberger10a} derived a Bernstein-type inequality, 
which turns out to be the same as the one 
for i.i.d.~processes modulo some logarithmic factor.
However, this modification seems to be significant
stronger than Rio's original $\tilde \phi$-mixing, so it remains unclear when the 
Bernstein-type inequality in \cite{Wintenberger10a} is applicable.
In addition, the $\tilde \phi$-mixing concept is still not large enough 
to cover many commonly considered  dynamical systems. 
To include such dynamical systems, \cite{Maume06a} proposed 
the $\ca C$-mixing coefficients, which further generalize $\tilde \phi$-mixing coefficients.

In this work, we establish a Bernstein-type inequality for  geometrically $\ca C$-mixing processes,
which, modulo a logarithmic factor and some constants, coincides with the classical one for i.i.d.~processes.
Using the techniques developed in \cite{HaSt14a}, we then 
derive an oracle inequality for generic regularized empirical risk minimization
and $\ca C$-mixing processes. 
We further apply this oracle inequality to a state-of-the-art learning method, namely support vector 
machines (SVMs) with Gaussian kernels. Here it turns out that for both, least squares and quantile regression,
we can recover the (essentially) optimal rates recently found for the i.i.d.~case, see \cite{EbSt11a},
when the data is generated by a geometrically $\ca C$-mixing process.
Finally, we establish an oracle inequality for the problem of forecasting an unknown dynamical system.
This oracle will make it possible to extend the purely asymptotic analysis in \cite{StAn09a}
to learning rates.


%

The rest of this work is organized as follows: 
In Section \ref{StochasticProcesses}, 
we recall the notion of  (time-reversed) $\ca C$-mixing  processes. We further illustrate this class of processes by some examples
and discuss the relation between $\ca C$-mixing   and other notions of mixing.
As the main result of this work, a Bernstein-type inequality for geometrically (time-reversed) $\ca C$-mixing
processes will be formulated in Section
\ref{bernstein}. There, we also compare our new Bernstein-type inequality to previously established 
concentration inequalities.
As an application of our Bernstein-type  inequality,
we will derive the oracle inequality for regularized risk minimization schemes 
in  Section \ref{applications}. We additionally derive 
learning rates for SVMs and an oracle inequality for forecasting certain dynamical systems.
All proofs can be found in the last section.

\section{$\ca C$-mixing processes} \label{StochasticProcesses}

\allowdisplaybreaks

In this section we recall two classes of stationary
stochastic processes 
called (time-reversed) $\ca C$-mixing processes
that have a certain decay of correlations for suitable pairs of functions.
We also present some examples of such processes including certain dynamical systems.

Let us begin by introducing some notations. 
In the following,  $(\Omega, \mathcal{A}, \mu)$ always denotes a probability space.
As usual, we write $L_p(\mu)$ for  the space of (equivalence classes of) measurable functions
$f : \Omega \to \R$ with finite $L_p$-norm $\|f\|_p$. 
It is well-known that $L_p(\mu)$ together with $\|f\|_p$ forms a Banach space.
Moreover, if  $\ca A '\subset \ca A$ is a sub-$\s$-algebra, then 
$L_1(\mathcal A', \mu)$ denotes the space of all $\ca A'$-measurable functions $f\in L_1(\mu)$. 
In the following, for a Banach space $E$, we write $B_E$ for its closed unit ball.

Given a semi-norm $\|\cdot\|$ on a vector space $E$ of bounded measurable
functions $f : Z \rightarrow \mathbb{R}$, we define the $\ca C$-Norm by
\begin{align} 
\|f\|_{\ca C} := \|f\|_{\infty} + \|f\|
\label{lambdanorm}
\end{align} 
and denote the space of all bounded $\ca C$-functions by 
\begin{align}\label{def:lZ} 
 \ca C(Z) := \bigl\{ f : Z \rightarrow \mathbb{R} \, \bigl| \, \|f\|_{\ca C} < \infty \bigr\}.
\end{align}
Throughout this work, we only consider the semi-norms $\|\cdot\|$ in (\ref{lambdanorm}) that satisfy the inequality
\begin{align} 
 \bigl\| e^f \bigr\|  \leq  \bigl\| e^f \bigr\|_{\infty} \|f\|
\label{expproperty}
\end{align}
for all $f\in \ca C(Z)$. We are mostly interested in the following examples of semi-norms satisfying
(\ref{expproperty}).

\begin{example}\label{example-null}
 Let $Z$ be an arbitrary set and suppose that we have
$\|f\| = 0$ for all $f : Z \rightarrow \mathbb{R}$.
Then, it is obviously to see that 
$\|e^f\| = \|f\| = 0$. Hence, (\ref{expproperty}) is satisfied. 
\end{example}

\begin{example}\label{example-BV1}
Let $Z \subset \mathbb{R}$ be an interval.
A function $f : Z \to \R$ is said to have bounded variation on $Z$
if its total variation $\|f\|_{BV(Z)}$ is bounded.
Denote by $BV(Z)$ the set of all functions of bounded variation.
It is well-known that
$BV(Z)$ together with 
$\|f\|_{\infty} + \|f\|_{BV(Z)}$
forms a Banach space. Moreover, we have (\ref{expproperty}), i.e.~we have for all 
$f \in \ca C(Z)$:
\begin{align*} 
 \bigl\| e^f \bigr\|_{BV(Z)}  \leq  \bigl\| e^f \bigr\|_{\infty} \|f\|_{BV(Z)}.
\end{align*}
\end{example}


\begin{example}\label{example-Hoelder}
Let $Z$ be a subset of $\mathbb{R}^d$ and
$C_b(Z)$ be the set of bounded continuous functions on $Z$. 
For $f \in C_b(Z)$ and $0 < \alpha \leq 1$ let
\begin{align*} 
 \|f\| := |f|_{\alpha} := \sup_{z \neq z'} \frac{|f(z) - f(z')|}{|z-z'|^{\alpha}}.
\end{align*}
Clearly,  $f$ is $\a$-H\"{o}lder continuous if and only if $|f|_\a < \infty$.
The collection of bounded, $\alpha$-H\"{o}lder continuous functions on $Z$ will be denoted by 
\begin{align*} 
 C_{b,\a}(Z) := \{ f \in C_b(Z) : |f|_\a < \infty \}.
\end{align*}
Note that, if $Z$ is compact, then $C_{b,\a}(Z)$ together with the norm
$\|f\|_{C_{b,\a}} := \|f\|_{\infty} + |f|_{\alpha}$ forms a Banach space.
Moreover, the inequality (\ref{expproperty}) is also valid
for $f \in C_{b,\alpha}(Z)$.
As usual, we speak of Lipschitz continuous functions if $\a=1$
and write $\mathrm{Lip}(Z) := C_{b,1}(Z)$.
\end{example}

\begin{example}
Let $Z \subset \mathbb{R}^d$ be an open subset.  
For a continuously differentiable function $f:Z\to \R$ we write  
\begin{align*} 
 \|f\| := \sup_{z\in   Z}|f'(z)|
\end{align*}
and $C^1(Z) := \bigl\{ f:Z\to \R \, |\, f \mbox{ continuously differentiable and } \inorm f + \snorm f < \infty    \bigr\}$.
It is well-known, that $C^1(Z)$ is a Banach space 
with respect to the norm $\inorm \cdot + \snorm \cdot$
and the chain rule gives 
\begin{align*}
\bigl\| e^f \bigr\|
= \bigl\| \bigl( e^f \bigr)' \bigr\|_{\infty}
= \bigl\| e^f \cdot f' \bigr\|_{\infty}
\leq \bigl\| e^f \bigr\|_{\infty} \| f'\|_{\infty}
= \bigl\| e^f \bigr\|_{\infty} \|f\|,
\end{align*}
for all $f\in C^1(Z)$, i.e.~(\ref{expproperty}) is satisfied. 
\end{example}

Let us now assume that we also have 
a measurable space $(Z, \mathcal{B})$ and a measurable map $\chi : \Omega \to Z$. 
Then $\sigma(\chi)$ denotes the smallest $\sigma$-algebra on $\Omega$ for which $\chi$ is measurable. 
Moreover, $\mu_{\chi}$ denotes the $\chi$-image measure of $\mu$, which is defined by
$\mu_{\chi}(B) := \mu(\chi^{-1}(B))$, $B \in \ca B$.

Let $\mathcal{Z} := (Z_n)_{n \geq 0}$ be a $Z$-valued stochastic process on $(\Omega, \mathcal A, \mu)$, 
and $\mathcal A_0^i$ and $\mathcal A_{i+n}^{\infty}$ be the $\sigma$-algebras generated by 
$(Z_0, \ldots, Z_i)$ and $(Z_{i+n}, Z_{i+n+1}, \ldots)$, respectively. 
The process  $\mathcal{Z}$ is called \emph{stationary} if 
$\mu_{(Z_{i_1+i}, \ldots, Z_{i_n+i})} = \mu_{(Z_{i_1}, \ldots, Z_{i_n})}$
for all $n, i, i_1, \ldots, i_n \geq 1$.
In this case, we always write $P := \mu_{Z_0}$.
Moreover, to define certain dependency coefficients for $\ca Z$, we denote, 
for $\psi, \varphi \in L_1(\mu)$ satisfying $\psi\varphi\in L_1(\mu)$ the correlation
of $\psi$ and $\varphi$ by 
\begin{align*}
\mathrm{cor} (\psi, \varphi) 
:= \int_{\Omega} \psi \cdot \varphi \, d \mu -
   \int_{\Omega} \psi \, d \mu \cdot \int_{\Omega} \varphi \, d \mu\, .
\end{align*} 
%
%
%
%
%
Several dependency coefficients for $\ca Z$ can be expressed by imposing 
restrictions on $\psi$ and $\varphi$.
The following definition,
which is   taken from \cite{Maume06a},
introduces the restrictions on $\psi$ and $\varphi$ we consider throughout this work.

\begin{definition}
Let $(\Omega, \mathcal A, \mu)$ be a probability space, $(Z, \mathcal{B})$ be a measurable space,
$\mathcal{Z}:=(Z_i)_{i \geq 0}$ be a $Z$-valued, stationary process on $\Omega$, and 
$\snorm\cdot_{\ca C}$ be defined by \eqref{lambdanorm} for some semi-norm $\snorm\cdot$.
Then, for $n \geq 0$, we define:
\begin{enumerate}
 \item[(i)]  the $\ca C$-mixing coefficients by
\begin{align} \label{phiLambda}
\phi_{\ca C}(\mathcal{Z}, n)
:= \sup \big \{ \mathrm{cor}(\psi, h \circ Z_{k+n}) : k \geq 0, \psi \in B_{L_1(\mathcal A_0^k, \mu)}, h \in B_{\ca C(Z)}
\big\}
\end{align} 
 \item[(ii)] the time-reversed $\ca C$-mixing coefficients by
\begin{align}\label{phiLambdarev}
\phi_{\ca C, \text{rev}}(\mathcal{Z}, n)
:= \sup \big\{ 
\mathrm{cor}(h \circ Z_k, \varphi) : 
k \geq 0, h \in B_{\ca C(Z)}, \varphi \in B_{L_1(\mathcal A_{k+n}^{\infty}, \mu)}  \big\}.
\end{align} 
\end{enumerate}

\end{definition} 

Let $(d_n)_{n \geq 0}$ be a strictly positive sequence converging to $0$. 
Then we say that $\mathcal{Z}$ is (time-reversed) $\ca C$-mixing with rate $(d_n)_{n \geq 0}$, 
if we have $\phi_{\ca C,(\text{rev})}(\mathcal{Z}, n) \leq d_n$ for all $n \geq 0$.
Moreover, if $(d_n)_{n \geq 0}$ is of the form
\begin{align}
d_n := c \exp \bigl( - b n^{\gamma} \bigr), ~~~~~~ n \geq 1, 
\label{dn}
\end{align}
for some constants $b > 0$, $c \geq 0$, and $\gamma > 0$, then $\ca Z$ is called 
\emph{geometrically} (time-reversed) $\ca C$-mixing.

Obviously, $\mathcal{Z}$ is  $\ca C$-mixing with rate $(d_n)_{n \geq 0}$, if and only if
for all $k,n \geq 0$, all $\psi\in L_1(\mathcal A_0^k, \mu)$, 
and all $h \in \ca C(Z)$, we have
\begin{align}
 \mathrm{cor}(\psi, h \circ Z_{k+n}) 
\leq \|\psi\|_{L_1(\mu)}  \|h\|_{\ca C}   \, d_n,
\label{decayd}
\end{align} 
or similarly, time-reversed $\ca C$-mixing with rate $(d_n)_{n \geq 0}$, if and only if
for all $k, n \geq 0$, all $h \in \ca C(Z)$, 
and all $\varphi\in L_1(\mathcal A_{k+n}^{\infty}, \mu)$, we have
\begin{align}
 \mathrm{cor} (h \circ Z_k, \varphi) 
\leq \|h\|_{\ca C} \|\varphi\|_{L_1(\mu)}  \, d_n.
\label{decaydrev}
\end{align}

In the rest of this section we consider examples of (time-reversed) $\ca C$-mixing processes.
To begin with, let us assume that $\ca Z$ is a stationary $\phi$-mixing process  \cite{Ibragimov62a}
with rate $(d_n)_{n \geq 0}$. By \cite[Inequality (1.1)]{Davydov68a}  we then
have 
\begin{align}
 \mathrm{cor}(\psi, \varphi) 
\leq \|\psi\|_{L_1(\mu)} \|\varphi\|_{L_\infty(\mu)} d_n, ~~~~ n \geq 1,
\label{decayphi}
\end{align}
for all $\mathcal A_0^k$-measurable $\psi \in L_1(\mu)$ and all 
$\mathcal A_{k+n}^{\infty}$-measurable $\varphi \in L_{\infty}(\mu)$.
By taking $\snorm\cdot_{\ca C}:=\inorm\cdot$ and $\varphi:= h\circ Z_{k+n}$, we then see that 
\eqref{decayd} is satisfied, i.e.~$\ca Z$ is $\ca C$-mixing with rate $(d_n)_{n \geq 0}$.
Finally, by similar arguments  we can  deduce that 
 time-reversed $\phi$-mixing processes \cite[Section 3.13]{Bradley07a} 
are also time-reversed $\ca C$-mixing with the same rate.
In other words we have found 
\begin{displaymath}
   \phi_{L_\infty(\mu)}(\mathcal{Z}, n) = \phi(\mathcal{Z}, n) \qquad \qquad 
\mbox{ and } \qquad \qquad \phi_{L_\infty(\mu), \text{rev}}(\mathcal{Z}, n) = \phi_{\text{rev}}(\mathcal{Z}, n).
\end{displaymath}

To deal with processes that are not $\a$-mixing \cite{Rosenblatt56a}, 
Rio \cite{Rio96a} introduced the following relaxation of $\phi$-mixing coefficients
\begin{align}\label{phi-mix}
\tilde \phi(\mathcal Z, n)
&:= \sup_{k\geq 0, \atop f \in BV_1} 
\big\| \mathbb{E} \bigl( f(Z_{k+n} )\big| \mathcal A_0^k \big) 
- \mathbb{E} f(Z_{k+n}) \big\|_{\infty} \\ \nonumber
&= \sup \big \{ \mathrm{cor}(\psi, h \circ Z_{k+n}) : k \geq 0, \psi \in B_{L_1(\mathcal A_0^k, \mu)}, h \in B_{BV(Z)} \bigr\}
\end{align}
and an analogous time-reversed coefficient
\begin{align*}
\tilde \phi_{\text{rev}}(\mathcal Z, n) 
&:= \sup_{k\geq 0, \atop f \in BV_1} 
\big\| \mathbb{E} \big( f(Z_k) \big| A_{k+n}^{\infty} \big) 
- \mathbb{E}   f(Z_k) \big\|_{\infty} \\
&= \sup \big \{ \mathrm{cor}(h \circ Z_{k}, \varphi) : k \geq 0, \varphi \in B_{L_1(\mathcal A_{k+n}^\infty, \mu)}, h \in B_{BV(Z)} \bigr\}\, ,
\end{align*}
where the two identities follow from  \cite[Lemma 4]{DePr05a}. In other words we have 
\begin{displaymath}
   \phi_{BV(Z)}(\mathcal{Z}, n) = \tilde \phi(\mathcal{Z}, n) \qquad \qquad \mbox{ and } \qquad \qquad \phi_{BV(Z), \text{rev}}(\mathcal{Z}, n) = \tilde \phi_{\text{rev}}(\mathcal{Z}, n)
\end{displaymath}
Moreover, \cite[p.~41]{DeDoLaRLoPr07a} shows that some
 uniformly expanding maps are $\tilde \phi$-mixing but not $\a$-mixing.
Figure \ref{Relationship1}  summarizes the relations between $\phi$, $\tilde \phi$, and $\ca C$-mixing.


%
%
%

\begin{figure}
\begin{center}
\begin{tikzpicture}
  \begin{scope}[fill opacity=0.4]
    \fill[blue]  \firstcircle;
    \fill[green] \thirdellipse;
    \fill[yellow] \fourthellipse;
  \end{scope}
  \begin{scope}[very thick,font=\large]
    \draw \firstcircle node at (6,-4) {\small{$\phi$-mixing}};
    \draw \thirdellipse node at (8.6,-4) {\small{$\tilde \phi$-mixing}};
    \draw \fourthellipse node at (10.8,-4) {\small{$\ca C$-mixing}};
  \end{scope}
\end{tikzpicture}
\caption{Relationship between $\phi$-, $\tilde \phi$-, and $\ca C$-mixing processes}
\label{Relationship1}
\end{center}
\end{figure}
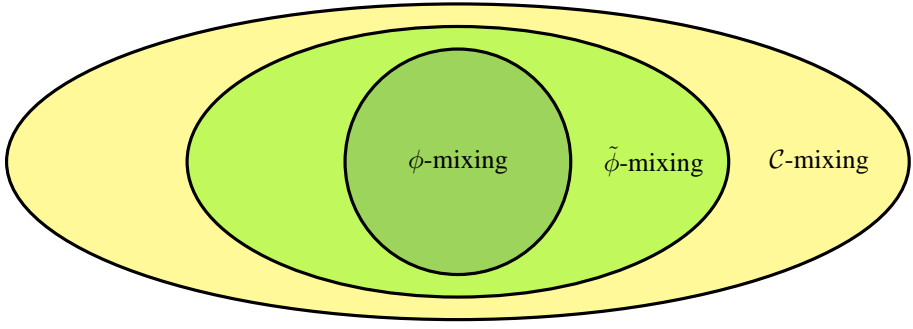

Our next goal is to relate $\ca C$-mixing to some well-known 
results on the decay of correlations for dynamical systems.
To this end, recall that 
 $(\Omega, \mathcal A, \mu, T)$ is a dynamical system, if 
 $T : \Omega \to \Omega$ is a measurable map satisfying 
$\mu(T^{-1} (A)) = \mu(A)$ for all $A \in \mathcal A$.
Let us consider the stationary stochastic process $\ca Z := (Z_n)_{n \geq 0}$ defined by 
$Z_n := T^n$ for $n\geq 0$.
Since $\ca A_{n+1}^{n+1} \subset \ca A_n^n$ for all $n\geq 0$, we conclude that 
$\ca A_{k+n}^\infty = \ca A_{k+n}^{k+n}$.
Consequently, $\varphi$ is $\ca A_{k+n}^\infty$-measurable, if and only if it is 
$\ca A_{k+n}^{k+n}$-measurable. Moreover   $\ca A_{k+n}^{k+n}$ is the $\s$-algebra 
generated by $T^{k+n}$, and hence $\varphi$ is $\ca A_{k+n}^{k+n}$-measurable, if and only if it is 
of the form $\varphi = g\circ T^{k+n}$ for some suitable, measurable $g:\Om\to\R$.
Let us now suppose that
$\snorm\cdot_{\ca C(\Om)}$ is defined by \eqref{lambdanorm} for some semi-norm $\snorm\cdot$.
For $h\in \ca C(\Om)$ we then find 
\begin{align*}
   \mathrm{cor} (h \circ Z_k, \varphi) 
    =
   \mathrm{cor} (h \circ Z_k, g\circ Z_{k+n}) 
   & = 
   \mathrm{cor} (h , g\circ Z_{n}) \\
   & =
   \int_{\Omega} h \cdot (g \circ T^n) \, d \mu -
   \int_{\Omega} h \, d \mu \cdot \int_{\Omega} g \, d \mu \\
   & = :  \mathrm{cor}_{T, n}(h,g)\,  .
\end{align*}
The next result shows that $\mathcal{Z}$ is time-reversed $\ca C$-mixing
even if we only have generic constants $C(h,g)$ in \eqref{decaydrev}.

\begin{theorem} \label{equivalencelambda} 
Let $(\Omega, \mathcal A, \mu, T)$ be a dynamical system
and the stochastic process $\ca Z := (Z_n)_{n \geq 0}$ be defined by 
$Z_n := T^n$ for $n\geq 0$. Moreover, Let
$\snorm\cdot_{\ca C}$ be defined by \eqref{lambdanorm} for some semi-norm $\snorm\cdot$.
Then, $\mathcal{Z}$ is time-reversed $\ca C$-mixing with rate $(d_n)_{n \geq 0}$ iff 
for all $h\in \ca C(\Om)$ and all $g \in L_1(\mu)$ there exists a constant $C(h,g)$ such that
\begin{align*}
 \mathrm{cor}_{T, n}(h,g) \leq C(h,g) d_n,
\,\,\,\,\,\,\,
n \geq 0.
\end{align*}
\end{theorem}

Thus, we see that 
$\ca Z$ is time-reversed $\ca C$-mixing, if 
$\mathrm{cor}_{T, n}(h,g)$ converges to zero for all $h\in \ca C(\Om)$ and $g\in L_1(\mu)$
with a rate that is independent of $h$ and $g$.

For concrete examples, let us first mention that
\cite{Maume06a} presents some discrete dynamical systems that are
time-reversed geometrically $\ca C$-mixing such as Lasota-Yorke
maps, uni-modal maps, piecewise expanding maps in higher dimension.
Here, the involved spaces are either $BV(Z)$ or $\mathrm{Lip}(Z)$.

In dynamical systems where chaos is weak, 
correlations often decay polynomially, i.e. 
the correlations satisfy
\begin{align}\label{ds-decay-poly}
|\mathrm{cor}_{T, n}(h, g)| \leq C(h, g) \cdot n^{- b} \, , \qquad\qquad n\geq 0,
\end{align}
for some constants $b > 0$ and $C(h, g) \geq 0$ depending on the functions $h$ and $g$. 
Young \cite{Young99a} developed a powerful method for studying correlations in systems 
with weak chaos where correlations decay at a polynomial rate for bounded $g$
and H\"{o}lder continuous $h$. 
Her method was applied to billiards with slow mixing rates, such as Bunimovich billiards, 
see \cite[Theorem 3.5]{BaMe08a}. For example, modulo some logarithmic factors
\cite{Markarian04a, ChZh05a} obtained 
\eqref{ds-decay-poly} with $b=1$ and $b=2$
for certain forms of  Bunimovich  billiards and
 H\"{o}lder continuous $h$ and $g$.
Besides these results, Baladi \cite{Baladi01a} also compiles
a list of \textquotedblleft parabolic\textquotedblright \ or 
\textquotedblleft intermittent\textquotedblright \ systems having a polynomial decay.

It is well-known that, if the functions $h$ and $g$ are sufficient smooth,
there exist dynamical systems where chaos is strong enough 
such that the correlations decay exponentially fast, that is, 
\begin{align}
|\mathrm{cor}_{T, n}(h, g)| 
\leq C(h, g) \cdot \exp \bigl( - b n^{\gamma} \bigr)\, , \qquad\qquad n\geq 0,
\label{CExpDecay}
\end{align}
for some constants $b > 0$, $\gamma > 0$, and $C(h, g) \geq 0$ depending on $h$ and $g$.
Again, Baladi \cite{Baladi01a} has listed some simple examples of dynamical systems 
enjoying \eqref{CExpDecay}
for analytic $h$ and $g$
such as the angle doubling map and the Arnold's cat map.
Moreover, for
continuously differentiable  $h$ and $g$,
\cite{Ruelle76a, Sinai72a}  proved
\eqref{CExpDecay}
for two closely related classes of systems, more precisely, 
$C^{1+\varepsilon}$ Anosov or 
the Axiom-A diffeomorphisms with Gibbs invariant measures and 
topological Markov chains, which are also known as subshifts of finite type,
see also \cite{Bowen75a}. These results 
were then extended by \cite{HoKe82a,Rychlik83a} 
to expanding interval maps with smooth invariant measures
for functions $h$ and $g$ of bounded variation. 
In the 1990s, similar results for H\"{o}lder continuous  $h$ and $g$ 
were proved for systems with somewhat weaker chaotic behavior which is characterized by nonuniform hyperbolicity, 
such as quadratic interval maps, see \cite{Young98a}, \cite{KeNo92a} and the H\'{e}non map \cite{BeYo00a}, and 
then extended to chaotic systems with singularities by \cite{Liverani95a} and specifically to Sinai billiards 
in a torus by \cite{Young98a,Chernov99a}. 
For some of these extensions,
such as smooth expanding dynamics, smooth nonuniformly hyperbolic
systems, and hyperbolic systems with singularities,
we refer to \cite{Baladi00a} as well. 
Recently, for $h$ of bounded variation and bounded $g$,
\cite{LuMe13a} obtained 
\eqref{CExpDecay}
for a 
class of piecewise smooth one-dimensional maps with
critical points and singularities. Moreover,
\cite{ArGaPa14a} has deduced \eqref{CExpDecay} for $h, g \in \mathrm{Lip}(Z)$ and a
suitable iterate of Poincar\'{e}'s first return map $T$ 
of a large class of singular hyperbolic flows.

\section{A Bernstein-type inequality} \label{bernstein}

In this section, we present the key result of this work, a Bernstein-type
inequality for stationary geometrically (time-reversed) $\ca C$-mixing process. 

\begin{theorem} \label{bernsteininequality}
Let $\mathcal Z := (Z_n)_{n \geq 0}$ be a $Z$-valued stationary 
geometrically (time-reversed) $\ca C$-mixing process on
$(\Omega, \mathcal A, \mu)$ 
with rate $(d_n)_{n \geq 0}$ as in (\ref{dn}),
$\snorm\cdot_{\ca C}$ be defined by \eqref{lambdanorm} for some semi-norm $\snorm\cdot$
satisfying (\ref{expproperty}),
and $P := \mu_{Z_0}$. 
Moreover, let $h \in \ca C(Z)$ with $\mathbb E_P h = 0$
and assume that there exist some $A > 0$, $B > 0$ , and $\sigma \geq 0$ such that
$\|h\| \leq A$, $\|h\|_{\infty} \leq B$, and $\mathbb E_P h^2 \leq \sigma^2$.
Then, for all $\varepsilon > 0$ and all
\begin{align}
n \geq n_0 := \max \left\{ \min \biggl\{ m \geq 3 : 
m^2 \geq \frac{808 c (3A + B)}{B} \text{ and } 
\frac{m}{(\log m)^{\frac{2}{\gamma}}} \geq 4 \biggr\},
e^{\frac{3}{b}} \right\},
\label{nzero}
\end{align}
we have
\begin{align}
\mu \left( \biggl\{ \omega \in \Omega : \frac{1}{n} \sum_{i=1}^n h \circ Z_i \geq \varepsilon \biggr\} \right) 
\leq 2 \exp \left( - \frac{n \varepsilon^2}{8 (\log n)^{\frac{2}{\gamma}} (\sigma^2 + \varepsilon B/3)} \right),
\label{bernsteininequality1}
\end{align}
or alternatively, for all $n \geq n_0$ and $\tau > 0$, we have
\begin{align}
\mu \left( \Biggl\{ \omega \in \Omega : 
\frac{1}{n} \sum_{i=1}^n h(Z_i(\omega)) 
\geq \sqrt{\frac{8 (\log n)^{\frac{2}{\gamma}} \sigma^2 \tau}{n}} 
     + \frac{8 (\log n)^{\frac{2}{\gamma}} B \tau}{3 n} \Biggr\} \right) 
\leq 2 e^{- \tau}.  
\label{bernsteininequality2}
\end{align}
\end{theorem}

Note that besides the additional logarithmic factor $4 (\log n)^{\frac{2}{\gamma}}$ and the constant $2$ in front of the exponential,
 \eqref{bernsteininequality1}  coincides with Bernstein's classical 
inequality for i.i.d.~processes.


In the remainder of this section, we compare Theorem \ref{bernsteininequality} with 
some other concentration inequalities 
for   non-i.i.d.~processes $\ca Z$. Here,  $\ca Z$ is real-valued and $h$ is the identity map if not 
specified otherwise.


%
%

\vspace{0.5cm}
\begin{example}
Theorem 2.3 in \cite{Baladi00a} shows that
smooth expanding systems on $[0, 1]$ have exponential
decay of correlations (\ref{decayd}).
Moreover, if, for such expanding systems,  the transformation $T$ is Lipschitz continuous and satisfies the conditions 
at the end of Section 4 in \cite{DePr05a} 
and the ergodic measure $\mu$ satisfies  \cite[condition (4.8)]{DePr05a}, then 
 \cite[Theorem 2]{DePr05a} shows that for all $\varepsilon \geq 0$ and $n \geq 1$,
the left-hand side of \eqref{bernsteininequality1} is bounded by 
\begin{align*}
\exp \left( - \frac{\varepsilon^2 n}{C} \right) 
\end{align*}
where $C$ is some constant independent of $n$. 
The same result has been proved in \cite[Theorem III.1]{CoMaSc02a} as well.
Obviously, this is a Hoeffding-type bound instead of a Bernstein-type one.
Hence, it is always larger than ours if 
the denominator of the exponent in (\ref{bernsteininequality1}) is smaller than $C$. 
\end{example}

\begin{example} 
For   dynamical systems 
with exponentially decreasing  $\tilde \phi$-coefficients, 
see \cite[condition (3.1)]{Wintenberger10a},
\cite[Theorem 3.1]{Wintenberger10a} provides a Bernstein-type inequality 
for $1$-Lipschitz functions 
$h:Z \to [-1/2,1/2]$
w.r.t. some metric $d$ on $Z$, 
in which the left-hand side of \eqref{bernsteininequality1} is bounded by 
\begin{align}
\exp \left( - \frac{C \varepsilon^2 n}{\sigma^2 + \varepsilon \log f(n)} \right)
\label{WinterBound}
\end{align}
for some constant $C$ independent of $n$ and
$f(n)$ being some function monotonically increasing in $n$. 
Note that modulo the logarithmic factor $\log f(n)$
the bound (\ref{WinterBound}) is the same as the one 
for i.i.d.~processes. 
Moreover, if $f(n)$ grows polynomially, cf.~\cite[Section 3.3]{Wintenberger10a}, then 
(\ref{WinterBound}) has the same asymptotic behaviour as our bound.
However, 
geometrically $\ca C$-mixing 
is weaker than  Condition (3.1) in \cite{Wintenberger10a}:
Indeed, the required exponential form of Condition (3.1) in \cite{Wintenberger10a}, 
i.e. 
\begin{align*}
 \sup_{k \geq 0} \tilde \phi (\mathcal{A}_0^k, \mathbf{Z}_{k+n}^{k+2n-1})
:= \sup_{k\geq 0} \sup_{f \in \mathcal{F}^n} 
\big\| \mathbb{E} \bigl( f(\mathbf{Z}_{k+n}^{k+2n-1}) \big| \mathcal A_0^k \big) 
- \mathbb{E} f(\mathbf{Z}_{k+n}^{k+2n-1}) \big\|_{\infty}
\leq c \cdot e^{- b n}
\end{align*}
for some $c, b > 0$ and all $n \geq 1$,
where $\mathbf{Z}_{k+n}^{k+2n-1} := (Z_{k+n}, \ldots, Z_{k+2n-1})$ and
$\mathcal{F}^n$ is the set of $1$-Lipschitz functions $f : Z^n \to [-\frac{1}{2}, \frac{1}{2}]$ 
w.r.t. the metric $d^n(x,y) := \frac{1}{n} \sum_{i=1}^n d(x_i, y_i)$,
implies 
\begin{align*}
 \sup_{k\geq 0} \sup_{f \in \mathcal{F}} 
\big\| \mathbb{E} \bigl( f(Z_{k+n} )\big| \mathcal A_0^k \big) 
- \mathbb{E} f(Z_{k+n}) \big\|_{\infty}
\leq c \cdot n e^{-bn} \leq c \cdot  e^{- \tilde{b}n}
\end{align*}
for some $c, \tilde{b} > 0$ and all $n \geq 1$, where
$\mathcal{F}$ is the set of $1$-Lipschitz functions $f : Z \to [-\frac{1}{2}, \frac{1}{2}]$ 
w.r.t. the metric $d$.  
In other words, processes satisfying Condition (3.1) in \cite{Wintenberger10a} 
are $\tilde \phi$-mixing, see \eqref{phi-mix},
which is stronger than geometrically $\ca C$-mixing, 
see again Figure \ref{Relationship1}. 
Moreover, our result holds for all $\g>0$, while \cite{Wintenberger10a} only considers the case $\g=1$.
%
\end{example}

\begin{example}
For an $\alpha$-mixing sequence of centered and bounded
random variables satisfying
$\alpha(n) \leq c \exp (- b n^{\gamma})$
for some constants $b > 0$, $c \geq 0$, and $\gamma > 0$,
\cite[Theorem 4.3]{MoMa96a} bounds the
 left-hand side of \eqref{bernsteininequality1}   by
\begin{align}
(1 + 4 e^{-2} c) \exp 
\left( - \frac{3 \varepsilon^2 n^{(\gamma)}}{6 \sigma^2 + 2 \varepsilon B} \right)
\,\,\,\,\,\,
\text{ with }
n^{(\gamma)} \asymp n^{\frac{\gamma}{\gamma+1}}
\label{alphaone}
\end{align}
for all $n \geq 1$ and all $\varepsilon > 0$. 
In general, this bound and our result are not comparable,
since not every $\alpha$-mixing process satisfies (\ref{decayd})
and conversely, not every process satisfying (\ref{decayd})
is necessarily $\alpha$-mixing, see Figure \ref{Relationship2}.
Nevertheless, for $\phi$-mixing processes,
it is easily seen that this bound is always worse than ours for a fixed $\gamma > 0$,
if $n$ is large enough. 
\end{example}

\begin{example} \label{alphaexample}
For an $\alpha$-mixing stationary sequence of centered and bounded
random variables satisfying $\alpha(n) \leq \exp(- 2 c n)$ for some $c > 0$, 
  \cite[Theorem 2]{MePeRi09a}  
bounds the   
   left-hand side of \eqref{bernsteininequality1}  by
\begin{align}
\exp \left( - \frac{C \varepsilon^2 n}{v^2  + \varepsilon B (\log n)^2 + n^{-1}B^2 } \right)\, ,
\label{alphatwo}
\end{align}
where $C>0$ is some  constant and
\begin{align} \label{vquadrat}
 v^2 := \sigma^2 + 2 \sum_{2 \leq i \leq n} |\mathrm{cov}(X_1, X_i)|\, .
\end{align}
By applying the covariance inequality for $\alpha$-mixing processes, 
see \cite[the corollary to Lemma 2.1]{Davydov68a}, 
we obtain $v^2 \leq C_{\delta} \|X_1\|_{2+\delta}^2$ for an arbitrary $\delta > 0$ and
a constant $C_{\delta}$   only depending on $\delta$. 
If the additional $\d>0$ is ignored, \eqref{alphatwo} has therefore   the same asymptotic behavior as
our bound. In general, however, the additional $\d$  does influence the asymptotic behavior. For example,
the oracle inequality we obtain in the next section would be slower by a factor of $n^\xi$, where $\xi > 0$ is arbitrary, if we used 
\eqref{alphatwo} instead. 
Finally, note that in general the bound \eqref{alphatwo} and ours are not comparable, see again Figure \ref{Relationship2}.
%

In particular, Inequality (\ref{alphatwo}) can be applied to geometrically $\phi$-mixing processes with $\g=1$.
By using the covariance inequality (1.1) for $\phi$-mixing processes in \cite{Davydov68a}, 
we can bound $v^2$ defined as in (\ref{vquadrat}) by $C \sigma^2$ with some constant $C$ independent of $n$. 
Modulo the term $n^{-1} B$ in the denominator, 
the  bound (\ref{alphatwo}) coincides with ours for geometrically $\phi$-mixing processes with $\g=1$.
However, our bound also holds for such processes with $\g\in (0,1)$.
%
\end{example}

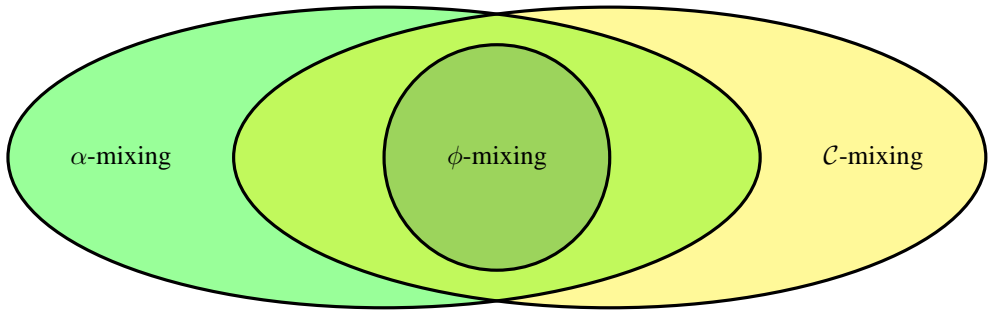
\begin{figure}
\begin{center}
\begin{tikzpicture}
  \begin{scope}[fill opacity=0.4]
    \fill[blue]  \firstcircle;
    \fill[green] \firstellipse;
    \fill[yellow] \secondellipse;
  \end{scope}
  \begin{scope}[very thick,font=\large]
    \draw \firstcircle node at (6,-4) {\small{$\phi$-mixing}};
    \draw \firstellipse node at (1,-4) {\small{$\alpha$-mixing}};
    \draw \secondellipse node at (11,-4) {\small{$\ca C$-mixing}};
  \end{scope}
\end{tikzpicture}
\caption{Relationship between $\alpha$-, $\phi$-, and $\ca C$-mixing processes}
\label{Relationship2}
\end{center}
\end{figure}

\begin{example}
For stationary, geometrically $\alpha$-mixing Markov chains 
with centered and bounded random variables, \cite{Adamczak08a} 
bounds the   
   left-hand side of \eqref{bernsteininequality1}  by
\begin{align}
 \exp \left( - \frac{n \varepsilon^2}{\tilde \sigma^2 + \varepsilon B \log n} \right),
\label{Adamczak}
\end{align}
where $\tilde \sigma^2 = \lim_{n \to \infty} \frac{1}{n} \mathrm{Var} \sum_{i=1}^n X_i$. 
By a similar argument as in Example \ref{alphaexample} we obtain
\begin{align*}
\mathrm{Var} \sum_{i=1}^n X_i
= n \sigma^2 + 2 \sum_{1 \leq i < j \leq n} |\mathrm{cov}(X_i, X_j)|
\leq n \sigma^2 + \tilde{C}_{\delta} n \|X_1\|_{2+\delta}^2
\end{align*}
for an arbitrary $\delta > 0$ and a constant $\tilde{C}_{\delta}$ depending 
only on $\delta$. Consequently we conclude that 
modulo some arbitrary small number $\delta > 0$ and the logarithmic factor $\log n$ 
instead of $(\log n)^2$, the bound (\ref{Adamczak}) coincides with ours.
Again, this bound and our result are not comparable, see Figure \ref{Relationship2}.
\end{example}

\begin{example}
For stationary, weakly dependent processes
of centered and bounded
random variables with $|\mathrm{cov}(X_1, X_n)| \leq c \cdot \exp(- b n)$ for some $c, b > 0$ and all $n \geq 1$,
  \cite[Theorem 2.1]{KaNe06a}  bounds the   left-hand side of \eqref{bernsteininequality1} by
\begin{align}
\exp \left( - \frac{\varepsilon^2 n}{C_1 + C_2 \varepsilon^{5/3} n^{2/3}} \right)
\label{KaNeBound}
\end{align}
where $C_1$ is some constant depending on $c$ and $b$, 
and $C_2$ is some constant depending on $c$, $b$, and $B$.
Note that the denominator in (\ref{KaNeBound}) is at least $C_1$, and therefore 
the bound (\ref{KaNeBound}) is  more of Hoeffding type.
\end{example}

\section{Applications to Statistical Learning} \label{applications}

\allowdisplaybreaks

In this section, we apply the Bernstein inequality from the last section to 
deduce oracle inequalities for some widely used  
learning methods and observations generated by a 
geometrically $\ca C$-mixing processes.
More precisely, in Subsection \ref{CRERMapplications}, we recall some basic concepts of statistical learning
and formulate an oracle inequality for learning methods that are based on (regularized) empirical risk minimization. Then, 
in the Subsection \ref{SVMapplications}, we illustrate this oracle inequality 
by deriving the learning rates for SVMs.
Finally, in Subsection \ref{FDS}, we present an oracle inequality 
for forecasting of dynamical systems.

\subsection{Oracle inequality for CR-ERMs} \label{CRERMapplications}

In this section, let $X$ always be a measurable space if not mentioned otherwise 
and $Y \subset \mathbb{R}$ always be a closed subset. 
Recall that in the (supervised) statistical learning,
our aim is to find a function $f : X \to \mathbb{R}$ 
such that for $(x, y) \in X \times Y$ the value $f(x)$ is 
a good prediction of $y$ at $x$. 
To evaluate the quality of such functions $f$, 
we need a loss function $L : X \times Y \times \mathbb{R} \to [0, \infty)$ that is measurable.
Following \cite[Definition 2.22]{StCh08a}, we say that a loss 
$L$ can be clipped at $M > 0$, if, for all 
$(x, y, t) \in X \times Y \times \mathbb{R}$, we have
\begin{align}
L(x, y, \wideparen{t} \,) \leq L(x, y, t), 
\label{clippedloss}
\end{align}
where $\wideparen{t}$ denotes the clipped value of $t$ at $\pm M$, that is
$\wideparen{t} := t$ if $t \in [-M, M]$, 
$\wideparen{t} := -M$ if $t < -M$, 
$\wideparen{t} := M$ if $t > M$.
Various often used loss functions can be clipped. 
For example, if $Y := \{ - 1, 1 \}$ and $L$ is a convex, margin-based loss represented by 
$\varphi : \mathbb R \rightarrow [0, \infty)$, 
that is $L(y, t) = \varphi(y t)$ for all $y \in Y$ and $t \in \mathbb R$, then $L$ can be 
clipped, if and only if $\varphi$ has a global minimum, see \cite[Lemma 2.23]{StCh08a}. 
In particular, the hinge loss, the least squares loss 
for classification, and the squared hinge loss can be clipped, 
but the logistic loss for classification and the AdaBoost loss cannot be clipped. 
Moreover, if $Y := [- M, M]$ and $L$ is a convex, 
distance-based loss represented by some $\psi: \mathbb R \rightarrow [0, \infty)$, 
that is $L(y, t) = \psi(y - t)$ for all $y \in Y$ and $t \in \mathbb R$, 
then $L$ can be clipped whenever $\psi(0) = 0$, see again \cite[Lemma 2.23]{StCh08a}. 
In particular, the least squares loss
\begin{align}
L(y, t) = (y - t)^2     
\label{lsloss} 
\end{align}
and the $\t$-pinball loss 
\begin{align}
L_\t(y,t) := \psi(y-t) = 
\begin{cases}
- (1 - \tau) (y - t), & \text{if } y - t < 0 \\ 
\tau (y - t), & \text{if } y - t \geq 0 
\end{cases}
\label{pbloss} 
\end{align}
used for quantile regression can be clipped, if the space of labels $Y$ is bounded. 

Now we summarize assumptions on the loss function $L$ that will be used
throughout this work.

\begin{assumption} \label{assumptionL}
 The loss function $L : X \times Y \times \mathbb R \rightarrow [0, \infty)$ 
can be clipped at some $M > 0$. Moreover, it is both bounded 
in the sense of $L(x, y, t) \leq 1$ and locally Lipschitz continuous, that is, 
\begin{align}
|L(x, y, t) - L(x, y, t')| \leq |t - t'|\, . 
\label{lipschitz}
\end{align}
Here both inequalites are supposed to hold for all $(x, y) \in X \times Y$ and $t, t' \in [- M, M]$.
Note that the former assumption can typically be enforced by scaling.
\end{assumption}

Given a loss function $L$ and an $f : X \rightarrow \mathbb R$, 
we often use the notation $L \circ f$ for the function $(x, y) \mapsto L(x, y, f(x))$. 
Our major goal is to have a small average loss for future unseen observations $(x, y)$. 
This leads to the following definition, see also \cite[Definitions 2.2 \& 2.3]{StCh08a}.

\begin{definition}
Let $L : X \times Y \times \mathbb R \rightarrow [0, \infty)$ be a loss function and 
$P$ be a probability measure on $X \times Y$. 
Then, for a measurable function $f : X \rightarrow \mathbb R$ the $L$-risk is defined by
\begin{align*}
\mathcal R_{L, P} (f) := \int\limits_{X \times Y} L(x, y, f(x)) \, d P(x, y).
\end{align*}
Moreover, the minimal $L$-risk
\begin{align*}
\mathcal R_{L, P}^* := \inf \{ \mathcal R_{L, P} (f) | f : X \rightarrow \mathbb R ~ \textrm{measurable} \}
\end{align*}
is called the Bayes risk with respect to $P$ and $L$. In addition, a measurable function $f_{L, P}^* : X \rightarrow \mathbb R$ 
satisfying $\mathcal R_{L, P} (f_{L, P}^*) = \mathcal R_{L, P}^*$ is called a Bayes decision function. 
\end{definition}

Informally, the goal of learning from a training set $D\in (X\times Y)^n$ is to find a decision function $f_D$ such that $\mathcal R_{L, P} (f_D)$ is close to the 
minimal risk $\mathcal R_{L, P}^*$. Our next goal is to formalize this idea. We begin with the following definition.

\begin{definition}
Let $X$ be a set and $Y \subset \mathbb R$ be a closed subset. A learning method $\mathcal L$ on $X \times Y$ maps every set 
$D \in (X \times Y)^n$, $n \geq 1$, to a function $f_D : X \rightarrow \mathbb R$.
\end{definition}

Let us now describe the learning algorithms we are interested in. To this end, we assume that we have a hypothesis set $\mathcal F$ consisting 
of bounded measurable functions $f : X \rightarrow \mathbb R$, which is pre-compact with respect to the supremum norm $\| \cdot \|_{\infty}$. 
Since $\mathcal F$ can be infinite, we need to recall the following, classical concept, which will enable us to approximate infinite $\mathcal F$ 
by finite subsets.

\begin{definition}
Let $(T, d)$ be a metric space and $\varepsilon > 0$. We call $S \subset T$ an $\varepsilon$-net of $T$ if for all $t \in T$ there exists an 
$s \in S$ with $d(s, t) \leq \varepsilon$. Moreover, the $\varepsilon$-covering number of $T$ is defined by
\begin{align*}
\mathcal N (T, d, \varepsilon) := \inf \left\{ n \geq 1 : \exists s_1, \ldots, s_n \in T ~ \textrm{such that} ~ T \subset \bigcup_{i=1}^n 
B_d(s_i, \varepsilon) \right\},
\end{align*}
where $\inf \emptyset := \infty$ and $B_d(s, \varepsilon) := \{ t \in T : d(t, s) \leq \varepsilon \}$ denotes the closed ball with center 
$s \in T$ and radius $\varepsilon$.
\end{definition}

Note that our hypothesis set $\mathcal F$ is assumed to be pre-compact, and hence for all $\varepsilon > 0$, 
the covering number $\mathcal N (\mathcal F, \| \cdot \|_{\infty}, \varepsilon)$ is finite.

In order to introduce our generic learning algorithms, we write
\begin{displaymath}
   D := \bigl( (X_1, Y_1), \ldots, (X_n, Y_n) \bigr) := (Z_1, \ldots, Z_n) \in (X \times Y)^n
\end{displaymath}
for a training set of length $n$ 
that is distributed according to the first $n$ components of 
the $X \times Y$-valued process
$\mathcal Z = (Z_i)_{i\geq1}$. Furthermore,
we write $D_n := \frac{1}{n} \sum_{i=1}^n \delta_{(X_i, Y_i)}$, 
where $\delta_{(X_i, Y_i)}$ denotes the (random) Dirac measure at $(X_i, Y_i)$. 
In other words, $D_n$ is the empirical measure associated to the data set $D$. Finally, 
the  risk of a function $f : X \to \mathbb R$ with respect to this measure
\begin{align*}
\mathcal R_{L, D_n} (f) = \frac{1}{n} \sum_{i=1}^n L(X_i, Y_i, f(X_i))
\end{align*}
is called the empirical $L$-risk.

With these preparations we can now introduce the class of learning methods we are interested in, 
see also \cite[Definition 7.18]{StCh08a}.

\begin{definition}\label{crerm}
Let $L : X \times Y \times \mathbb{R} \to [0, \infty)$ be a loss that can be clipped at some $M > 0$, 
$\mathcal F$ be a hypothesis set, that is, a set of measurable functions $f : X \rightarrow \mathbb R$, with $0 \in \mathcal F$, and
$\Upsilon$ be a regularizer on $\ca F$, that is, 
a function $\Upsilon : \mathcal F \rightarrow [0, \infty)$ with $\Upsilon(0) = 0$.
Then, for $\delta \geq 0$, a learning method whose decision functions $f_{D_n, \Upsilon} \in \mathcal F$ satisfy
\begin{align}
\Upsilon (f_{D_n, \Upsilon}) + \mathcal R_{L, D_n} (\wideparen{f}_{D_n, \Upsilon}) \leq \inf_{f \in \mathcal F} \left( \Upsilon(f) +
\mathcal R_{L, D_n} (f) \right) + \delta \label{deltaCRERM}
\end{align}
for all $n \geq 1$ and $D_n \in (X \times Y)^n$ is called $\delta$-approximate clipped regularized empirical risk minimization 
($\delta$-CR-ERM)
with respect to $L$, $\mathcal F$, and $\Upsilon$.

Moreover, in the case $\delta = 0$, we simply speak of clipped regularized empirical risk minimization (CR-ERM).
\end{definition}

Note that on the right-hand side of (\ref{deltaCRERM}) the unclipped loss is considered, 
and hence CR-ERMs do not necessarily minimize the regularized clipped empirical risk 
$\Upsilon (\cdot) + \mathcal R_{L, D_n} (\wideparen{\cdot})$. 
Moreover, in general CR-ERMs do not minimize the regularized risk 
$\Upsilon (\cdot) + \mathcal R_{L, D_n} (\cdot)$ either, 
because on the left-hand side of (\ref{deltaCRERM}) the clipped function is considered. 
However, if we have a minimizer of the unclipped regularized risk, then it automatically satisfies 
(\ref{deltaCRERM}). 
As an example of CR-ERMs, SVMs will be discussed in Section \ref{SVMapplications}.

Before we present the oracle inequality for $\delta$-CR-ERMs, we need to introduce a few more notations. 
Let $\mathcal F$ be a hypothesis set in the sense of Definition \ref{crerm}. For
\begin{align}
r^* := \inf_{f \in \mathcal F} \Upsilon(f) + \mathcal R_{L,P}(\wideparen{f} \,) - \mathcal R_{L,P}^* 
\label{rstar}
\end{align}
and $r > r^*$, we write
\begin{align}
\mathcal F_r := \left\{ f \in \mathcal F : \Upsilon(f) + \mathcal R_{L,P}(\wideparen{f}\,) - \mathcal R_{L,P}^* \leq r \right\}.
\label{Fr}
\end{align}
Then we have $r^* \leq 1$, since $L(x, y, 0) \leq 1$, $0 \in \mathcal F$, and $\Upsilon(0) = 0$.  
Furthermore, 
we assume that we have  a monotonic decreasing sequence $(A_r)_{r\in(0,1]}$
such that 
\begin{align}
 \|L \circ \wideparen{f}\| 
\leq A_r 
\,\,\,\,
\text{ for all }
f \in \mathcal F_r
\text{ and }
r \in (0, 1]\, ,
\label{lipfr}
\end{align}
where $\snorm \cdot$ is a semi-norm satisfying  \eqref{expproperty}.
Because of the definition (\ref{Fr}), 
it is easily to conclude that $\|L \circ \wideparen{f}\| \leq A_1$ 
for all $f \in \mathcal F_r$ and $r \in (0, 1]$.
Finally, we assume that 
there exists
a function $\varphi: (0, \infty) \rightarrow (0, \infty)$ 
and a $p \in (0, 1]$ such that,
for all $r > 0$ and $\varepsilon > 0$, 
we have 
\begin{align}
\ln \mathcal N (\mathcal F_r, \| \cdot \|_{\infty}, \varepsilon) \leq \varphi(\varepsilon) r^p.
\label{coveringnumber}
\end{align}
Note that there are actually many hypothesis sets  satisfying Assumption 
(\ref{coveringnumber}), see \cite[Section 4]{HaSt14a} for some examples.

Now the oracle inequality for CR-ERMs reads as follows:

\begin{theorem} \label{oracleinequality}
Let $\mathcal Z := (Z_n)_{n \geq 0}$ be a $Z$-valued stationary 
geometrically (time-reversed) $\ca C$-mixing process on
$(\Omega, \mathcal A, \mu)$ 
with rate $(d_n)_{n \geq 0}$ as in (\ref{dn}),
$\snorm\cdot_{\ca C}$ be defined by \eqref{lambdanorm} for some semi-norm $\snorm\cdot$
satisfying (\ref{expproperty}),
and $P := \mu_{Z_0}$.
Moreover, let $L$ be a loss satisfying Assumption \ref{assumptionL}.
In addition, assume that there exist a Bayes decision function $f_{L, P}^*$ 
and constants $\vartheta \in [0, 1]$ and 
$V \geq 1$ such that
\begin{align}
\mathbb E_P (L \circ \wideparen{f} - L \circ f_{L, P}^*)^2 
\leq V \cdot \left( \mathbb E_P (L \circ \wideparen{f} - L \circ f_{L, P}^*) \right)^{\vartheta}, 
~~~~~~  f \in \mathcal F, 
\label{variancebound}
\end{align}
where $\mathcal F$ is a hypothesis set with $0 \in \mathcal F$. 
We define $r^*$, $\mathcal F_r$, and $A_r$ by 
(\ref{rstar}), (\ref{Fr}), and (\ref{lipfr}), respectively 
and assume that (\ref{coveringnumber}) is satisfied. 
Finally, let $\Upsilon : \mathcal F \rightarrow [0, \infty)$ be a regularizer with $\Upsilon(0) = 0$, 
$f_0 \in \mathcal F$ be a fixed function, and
$A_0, A^* \geq 0$, $B_0 \geq 1$ be constants such that 
$\|L \circ f_0\| \leq A_0$, 
$\|L \circ \wideparen{f}_0\| \leq A_0$, 
$\|L \circ f_{L,P}^*\| \leq A^*$ and
$\| L \circ f_0 \|_{\infty} \leq B_0$. 
Then, for all fixed $\varepsilon > 0$, $\delta \geq 0$, $\tau \geq 1$, and
\begin{align}
n \geq n_0^* := \max \left\{ \min \left\{ m \geq 3 : 
m^2 \geq K \text{ and } 
\frac{m}{(\log m)^{\frac{2}{\gamma}}} \geq 4 \right\},
e^{\frac{3}{b}} \right\}
\label{nzerostar}
\end{align}
with $K = 1212 c (4 A_0 + A^* + A_1 + 1)$, and $r \in (0, 1]$ satisfying
\begin{align}
r \geq \max \left\{ 
\left( \frac{c_V (\log n)^{\frac{2}{\gamma}} (\tau + \varphi(\varepsilon/2) 2^p r^p)}{n} \right)^{\frac{1}{2-\vartheta}}, 
\frac{20 (\log n)^{\frac{2}{\gamma}} B_0 \tau}{n}, r^* \right\} 
\label{minradius}
\end{align}
with $c_V := 512 (12 V + 1)/3$, 
every learning method defined by (\ref{deltaCRERM}) satisfies with probability $\mu$ not less than $1 - 16 e^{-\tau}$:
\begin{align}
\Upsilon(f_{D_n, \Upsilon}) + \mathcal R_{L, P} (\wideparen{f}_{D_n, \Upsilon}) - \mathcal R_{L, P}^*
< 2 \Upsilon(f_0) + 4 \mathcal R_{L, P}(f_0) - 4 \mathcal R_{L, P}^* + 4 r + 5 \varepsilon + 2 \delta.
\label{oracleinequalityy}
\end{align}
\end{theorem}

Let us briefly discuss the variance bound (\ref{variancebound}). 
For example, if $Y = [- M, M]$ and $L$ is the least squares loss, 
then it is well-known that (\ref{variancebound}) is satisfied for 
$V := 16 M^2$ and $\vartheta = 1$, see e.g. \cite[Example 7.3]{StCh08a}. 
Moreover, under some assumptions on the distribution $P$, 
\cite{StCh11a} established a variance bound of the form (\ref{variancebound}) 
for the   pinball loss used for quantile regression. 
In addition, for the hinge loss, (\ref{variancebound}) is satisfied for 
$\vartheta := q / (q + 1)$, if Tsybakov's noise assumption 
\cite{Tsybakov04a} holds for 
$q$, see \cite[Theorem 8.24]{StCh08a}. Finally, based on \cite{BlLuVa03a}, \cite{Steinwart09a} established a variance bound with 
$\vartheta = 1$ for the earlier mentioned clippable modifications of strictly convex, twice continuously differentiable margin-based loss 
functions. 

One might wonder, why the constants $A_0$ and $B_0$ are necessary in Theorem \ref{oracleinequality}, 
since it appears to add further complexity. 
However, a closer look reveals that the constants $A_1$ and $B$ are the bounds
for functions of the form $L \circ \wideparen{f}$, 
while $A_0$ and $B_0$ are valid for the function $L \circ f_0$ for an {\em unclipped} $f_0 \in \mathcal F$. 
Since we do not assume that all $f \in \mathcal F$ satisfy $\wideparen{f} = f$, 
we conclude that in general $A_0$ and $B_0$ are necessary.

The following lemma 
 shows that the required bounds on $\snorm{L\circ f}$ do hold for specific loss functions, if  $\ca C = \mathrm{Lip}$ and 
the involved functions $f\in \ca F$ are Lipschitz, too.

\begin{lemma} \label{lipschitzlemma}
Let $(X, d)$ be a metric space, $Y \subset [-M, M]$ with $M > 0$.
Moreover, let 
 $f : X \to \mathbb R$ be a bounded, Lipschitz continuous function. Then the following statements hold true:
\begin{enumerate}
 \item[(i)] For the least square loss $L$, see (\ref{lsloss}), we have
\begin{align*}
 |L \circ f|_1 \leq 2\sqrt{2} \left( M + \|f\|_{\infty} \right) (1 + |f|_1).
\end{align*}
 \item[(ii)] For the $\t$-pinball loss $L$, see  (\ref{pbloss}), we have
\begin{align*}
 |L \circ f|_1 \leq \sqrt{2} (1 + |f|_1).
\end{align*}
\end{enumerate}
\end{lemma}

\subsection{Learning rates for SVMs} \label{SVMapplications}

Let us begin by briefly recalling SVMs, see \cite{StCh08a} for details. 
To this end, let $X$ be a measurable space, $Y := [-1, 1]$ and $k$ be a 
measurable (reproducing) 
kernel on $X$ with reproducing kernel Hilbert space (RKHS) $H$. 
Given a regularization parameter $\lambda > 0$ and a convex loss 
$L$, SVMs find the unique solution
\begin{align}
f_{D_n, \lambda} = \textrm{arg} \min_{f \in H} \left( \lambda \| f \|_H^2 + \mathcal R_{L, D_n} (f) \right).
\label{svmsolution}
\end{align}
In particular, SVMs using the least-squares loss (\ref{lsloss}) are called least-squares SVMs (LS-SVMs), while SVMs using the 
$\t$-pinball loss (\ref{pbloss}) are called SVMs for quantile regression. 

Note that SVM decision functions (\ref{svmsolution}) satisfy (\ref{deltaCRERM}) 
for the regularizer $\Upsilon := \lambda \|\cdot\|_H^2$ and 
$\delta : = 0$. In other words, SVMs are CR-ERMs. 
Consequently we can use the oracle inequality in Theorem \ref{oracleinequality}
to derive the learning rates for SVMs.

Assumption \ref{assumptionL} implies that 
\begin{align*}
 \lambda \| f_{D_n, \lambda} \|_H^2
\leq \lambda \| f_{D_n, \lambda} \|_H^2 + \mathcal R_{L, D_n} (f)
= \min_{f \in H} \left( \lambda \| f \|_H^2 + \mathcal R_{L, D_n} (f) \right)
\leq \mathcal R_{L, D_n} (0)
\leq 1.
\end{align*}
In other words, for a fix $\lambda > 0$, we have 
\begin{align}
 f_{D_n, \lambda} \in \lambda^{-1/2} B_H, 
\label{hyposet}
\end{align}
where $B_H$ denotes the closed unit ball of the RKHS $H$. 

In the following, we are mainly interested in the commonly used Gaussian RBF kernels 
$k_{\sigma} : X \times X \to \mathbb R$ defined by
\begin{align*}
k_{\sigma} (x, x') := \exp \left( - \frac{\|x - x'\|_2^2}{\sigma^2} \right), 
\,\,\,\,\,\,\,\,\,\,\,\,\,\,\,\,
x, x' \in X,
\end{align*}
where $X \subset \mathbb R^d$ is a nonempty subset and $\sigma > 0$ is a free parameter called the width. 
We write $H_{\sigma}$ for the corresponding RKHSs, which are described in some detail in \cite{StHuSc06a}.
The entropy numbers for Gaussian kernels \cite[Theorem 6.27]{StCh08a} and the
equivalence of covering and entropy numbers \cite[Lemma 6.21]{StCh08a} yield that
\begin{align}
\ln \mathcal N (B_{H_\sigma}, \| \cdot \|_{\infty}, \varepsilon) 
\leq a \sigma^{-d} \varepsilon^{- 2 p}, 
\,\,\,\,\,\,\,\,\,\,\,\,\,\,\,\,
\varepsilon > 0,
\label{covernumber3}
\end{align}
for some constants $a > 0$ and $p \in (0, 1)$.

Because of (\ref{hyposet}), we can choose the hypothesis set as
$\mathcal F = \lambda^{-1/2} B_{H_\sigma}$.
Then the definition (\ref{Fr}) implies that 
$\mathcal F_r \subset r^{1/2} \lambda^{-1/2} B_{H_\sigma}$ and consequently we have
\begin{align*}
\ln \mathcal N (\mathcal F_r, \| \cdot \|_{\infty}, \varepsilon) 
\leq a \sigma^{-d} \lambda^{-p} \varepsilon^{- 2 p} r^{p}, 
\end{align*}
and thus, for the function $\varphi$ in Theorem \ref{oracleinequality}, we can choose
\begin{align}
\varphi(\varepsilon) := 
a \sigma^{-d} \lambda^{- p} \varepsilon^{- 2 p}.
\label{varphi}
\end{align}

Now, with some additional assumptions below, we can use the oracle inequality in Theorem \ref{oracleinequality}
to derive the learning rates for the SVMs using Gaussian kernels. 
In the following, $B_{2s, \infty}^t$ denotes the usual Besov space with the smoothness parameter $t$, 
more details see \cite[Section 2]{EbSt11a}.

\begin{theorem}[Least Square Regression with Gaussian Kernels]   \label{LSSVMgaussiankernels}
Let $Y := [- M, M]$ for $M > 0$, and $P$ be a distribution on $\mathbb R^d \times Y$ such that 
$X := \mathrm{supp} P_X \subset B_{\ell_2^d}$ is a bounded domain with $\mu(\partial X) = 0$,
where $B_{\ell_2^d}$ denotes the closed unit ball of 
$d$-dimensional Euclidean space $\ell_2^d$.
Furthermore, let $P_X$ be absolutely continuous w.r.t. the Lebesgue measure $\mu$ on $X$ with associated density 
$g : \mathbb R^d \rightarrow \mathbb R$ such that $g \in L_{q}(X)$ for some $q \geq 1$.
Moreover, let $f^*_{L, P} : \mathbb R^d \rightarrow \mathbb R$ be a Bayes decision function such that 
$f^*_{L, P} \in L_2(\mathbb R^d) \cap \mathrm{Lip}(\mathbb R^d)$ as well as
$f^*_{L, P} \in B_{2s, \infty}^t$ for some $t \geq 1$ and $s \geq 1$ with $\frac{1}{q} + \frac{1}{s} = 1$.
Then, for all $\xi > 0$, the LS-SVM using Gaussian RKHS $H_{\sigma}$ and
\begin{align}
\lambda_n = n^{- 1} ~~~~ \textrm{and} ~~~~ 
\sigma_n = n^{- \frac{1}{2 t + d}} \ ,    \label{lambdagamma}
\end{align}
learns with rate 
\begin{align}
n^{- \frac{2 t}{2 t + d} + \xi}\, .
\label{lrlssvmgaussian}
\end{align}
\end{theorem}

It turns out that, modulo the arbitrarily small $\xi>0$, these learning rates are  optimal, 
see e.g. \cite[Theorem 13]{StHuSc09b} or \cite[Theorem 3.2]{GyKoKrWa02}.

To achieve these rates, however, we need to set $\lambda_n$ and $\sigma_n$ as in (\ref{lambdagamma}), which in turn requires us to 
know $t$.
Since in practice we usually do not know these values nor their existence,  
we can use the training/validation approach TV-SVM, see e.g.~\cite[Chapters 6.5, 7.4, 8.2]{StCh08a}, to
achieve the same rates adaptively, i.e.~without knowing $t$. To this end, let $\Lambda := \left( \Lambda_n \right)$ and 
$\Sigma := \left( \Sigma_n \right)$ be sequences of finite subsets $\Lambda_n, \Sigma_n \subset (0, 1]$ 
such that $\Lambda_n$ is an $\epsilon_n$-net of $(0,1]$ and 
$\Sigma_n$ is an $\delta_n$-net of $(0,1]$ with $\epsilon_n\leq n^{-1}$ 
and $\delta_n\leq n^{-\frac{1}{2+d}}$. Furthermore, assume that the 
cardinalities $\left|\Lambda_n\right|$ and $\left|\Sigma_n\right|$ grow polynomially in $n$.
For a data set $D := \left( \left( x_1, y_1 \right), \ldots, \left( x_n, y_n \right) \right)$, we define
\begin{align*}
D_1 & := \left( \left( x_1, y_1 \right), \ldots, \left( x_m, y_m \right) \right) \\
D_2 & := \left( \left( x_{m+1}, y_{m+1} \right), \ldots, \left( x_n, y_n \right) \right)
\end{align*}
where $m := \left\lfloor \frac{n}{2} \right\rfloor + 1$ and $n \geq 4$. We will use $D_1$ as a training set by computing the SVM decision 
functions
\begin{align*}
f_{\mathrm{D}_1, \lambda, \sigma} := \arg\min_{f \in H_\sigma} \lambda \left\| f \right\|^2_{H_\sigma} 
+ \mathcal{R}_{L, \mathrm{D}_1} \left(f\right), 
\qquad \qquad 
\left( \lambda, \sigma \right) \in \Lambda_n \times \Sigma_n  
\end{align*}
and use $D_2$ to determine $\left(\lambda,\sigma\right)$ by choosing a 
$\left( \lambda_{\mathrm{D}_2}, \sigma_{\mathrm{D}_2} \right) \in \Lambda_n \times \Sigma_n$ such that
\begin{align*}
\mathcal{R}_{L,\mathrm{D}_2} \left( \wideparen{f}_{\mathrm{D}_1, \lambda_{\mathrm{D}_2}, \sigma_{\mathrm{D}_2}} \right) 
= \min_{\left(\lambda,\sigma\right)\in\Lambda_n\times\Sigma_n} 
\mathcal{R}_{L,\mathrm{D}_2} \left( \wideparen{f}_{\mathrm{D}_1,\lambda,\sigma}\right)\ . 
\end{align*}
Then, analogous to the proof of Theorem 3.3 in \cite{EbSt11a} we can show that for all $\zeta>0$ and $\x>0$, the TV-SVM 
producing the decision functions $f_{\mathrm{D}_1,\lambda_{\mathrm{D}_2},\sigma_{\mathrm{D}_2}}$
with the above learning rates (\ref{lrlssvmgaussian}).

The following remark discusses learning rates for SVMs for quantile regression. For 
more information on such SVMs we refer to \cite[Section 4]{EbSt11a}.

\begin{remark}[Quantile Regression with Gaussian Kernels]   \label{svmquantile}
Let $Y := [- 1, 1]$, and $P$ be a distribution on $\mathbb R^d \times Y$ such that 
$X := \mathrm{supp} P_X \subset B_{\ell_2^d}$ be a domain. 
Furthermore, we assume that, for $P_X$-almost all $x \in X$, 
the conditional measure $P(\cdot|x)$ is absolutely continuous w.r.t.~the Lebesgue measure on $Y$
and the conditional density $h(\cdot, x)$ of $P(\cdot|x)$ is bounded from $0$ and $\infty$, 
see also \cite[Example 4.5]{EbSt11a}. 
Moreover, let $P_X$ be absolutely continuous w.r.t.~the Lebesgue measure on $X$ with associated density
$g \in L_{u}(X)$ for some $u \geq 1$.
For $\tau \in (0, 1)$, let $f^*_{\tau, P} : \mathbb R^d \to \mathbb R$ be a conditional $\tau$-quantile function that satisfies 
$f^*_{\tau, P} \in L_2(\mathbb R^d) \cap \mathrm{Lip}(\mathbb R^d)$.
In addition, we assume that  $f^*_{\tau, P} \in B_{2s, \infty}^t$ for some $t \geq 1$ and $s \geq 1$ such that 
$\frac{1}{s} + \frac{1}{u} = 1$. 
Then \cite[Theorem 2.8]{StCh11a} yields a variance bound of the form 
\begin{align*}
\mathbb E_{P} (L_{\tau} \circ \wideparen{f} - L_{\tau} \circ f^*_{\tau,P})^2 \leq
V\cdot \mathbb{E}_P (L_{\tau} \circ \wideparen{f} -L_{\tau} \circ f^*_{\tau,P}) \, ,
\end{align*}
for all $f:X\to \R$, where $V$ is a suitable constant and $L_\t$ is the $\t$-pinball loss.
Similar arguments to Theorem \ref{LSSVMgaussiankernels} shows that
the essentially optimal learning rate (\ref{lrlssvmgaussian}) 
can be achieved as well.
Note that the rate (\ref{lrlssvmgaussian}) is for the excess $L_\t$-risk, but since
\cite[Theorem 2.7]{StCh11a} shows 
\begin{displaymath}
   \snorm{\wideparen f-f_{\t,P}^*}_{L_2(P_X)}^2 \leq c \bigl(\RP {L_\t}{\wideparen f\,} - \RPB {L_\t} \bigr)
\end{displaymath}
for some constant $c>0$ and all $f:X\to \R$, we actually obtain the same rates for 
$\snorm{\wideparen f-f_{\t,P}^*}_{L_2(P_X)}^2$. Last but not least, optimality and adaptivity 
can be discussed along the lines of LS-SVMs.
\end{remark}

\subsection{Forecasting of dynamical systems} \label{FDS}

In this section, 
we proceed with the study of the forecasting problem of dynamical systems
considered in \cite{StAn09a}. 
First, let us recall some basic notations and assumptions.
Let $\Om$ be a compact subset of $\mathbb R^d$,
$(\Omega, \mathcal A, \mu, T)$ be a dynamical system,
and $S_0 \in \Om$ be a random variable describing the true but unknown state at time $0$. 
Moreover, for $E > 0$,
assume that all observations of 
the stochastic process described by the sequence 
$\mathcal T := (T^n)_{n \geq 0}$
are additively corrupted by some i.i.d., $[-E, E]^d$-valued 
noise process $\mathcal E = (\varepsilon_n)_{n \geq 0}$ 
defined on the probability space $(\Theta, \mathcal{C}, \nu)$
which is (stochastically) independent of $\mathcal T$.
It follows that all possible observations of the system at time $n \geq 0$ are of the form
\begin{align}
X_n = T^n(S_0) + \varepsilon_n. 
\label{ds}
\end{align}
In other words, the process that generates the noisy observations (\ref{ds}) is 
$(T^n(S_0) + \varepsilon_n)_{n \geq 0}$. 
In particular, a sequence of observations 
$(X_0, \ldots, X_n)$ generated by this process is of the form (\ref{ds}) 
for a conjoint initial state $S_0$.

Now, given an observation of the process $\mathcal T := (T^n)_{n \geq 0}$ at some arbitrary 
time, our goal is to forecast the next \textit{observable} state. To do so,
we will use the training set
\begin{align*}
\boldsymbol{D}_{\!n} 
& = 
\left( \left( X_0, X_1 \right), \ldots, \left( X_{n-1}, X_n \right) \right) 
\\
& = \left( \left( S_0 + \varepsilon_0, T(S_0) + \varepsilon_1 \right), \ldots, 
\left( T^{n-1}(S_0) + \varepsilon_{n-1}, T^n(S_0) + \varepsilon_n \right) \right)
\end{align*} 
whose input/output pairs are consecutive observable states. 
In other words, our goal is to use $\boldsymbol{D}_{\!n}$ to build a forecaster 
\begin{align*}
\boldsymbol{f}_{\!\! \boldsymbol{D}_n} : \mathbb R^d \rightarrow \mathbb R^d 
\end{align*}
whose average forecasting performance on future noisy observations is as small as possible. 
In order to render this goal, we will use the forecaster
\begin{align}
\boldsymbol{f}_{\!\! \boldsymbol{D}_n} 
:= \left( f_{\!\! \boldsymbol{D}_{\!n}^{(1)}}, \ldots, f_{\!\! \boldsymbol{D}_{\!n}^{(d)}} \right),
\label{forecasterd}
\end{align} 
where $f_{\!\! \boldsymbol{D}_{\!n}^{(j)}}$ is the forecaster obtained by using the training set
\begin{align*}
\boldsymbol{D}_{\!n}^{(j)} 
:= ((X_0, \pi_j(X_1)), \ldots, (X_{n-1}, \pi_j(X_n)))
\end{align*}
which is obtained by projecting the output variable of $\boldsymbol{D}_{\!n}$ onto its $j$th-coordinate via
the coordinate projection $\pi_j : \mathbb R^d \to \mathbb R$.

In other words, we build the forecaster
$\boldsymbol{f}_{\!\! \boldsymbol{D}_{\!n}}$ by training separately $d$ different decision functions 
on the training sets $\boldsymbol{D}_{\!n}^{(1)}, \ldots, \boldsymbol{D}_{\!n}^{(d)}$.
These problems can be considered as the (supervised) statistical learning problems 
formulated in Subsection \ref{CRERMapplications} with the help of the following Notations.

For $E > 0$ and a fixed $j \in \{ 1, \ldots, d \}$, 
we write $X := K + [- E, E]^d$, $Y := \pi_j(X)$ and $Z := X \times Y$. 
Moreover, we define the $X \times Y$-valued process
$\mathcal Z = (Z_n)_{n \geq 0} = (X_n, Y_n)_{n \geq 0}$ on 
$(K \times \Theta, \mathcal{B} \otimes \mathcal{C}, \mu \otimes \nu)$ by 
$X_n := T^n + \varepsilon_n$ and
$Y_n := \pi_j(T^{n+1} + \varepsilon_{n+1})$.
In addition, we write $P := (\mu \otimes \nu)_{(X_0, Y_0)}$.
Obviously, if the stochastic process $\mathcal{T}$ is $\ca C$-mixing 
and the noise process $\mathcal{E}$ is i.i.d, then
the stochastic processes 
\begin{align*}
\mathcal{Z} = (X_n, Y_n)_{n \geq 0} 
= (T^n(S_0) + \varepsilon_n, \pi_j(T^{n+1}(S_0) + \varepsilon_{n+1}))_{n \geq 0}
\end{align*}
is $\ca C$-mixing as well.

To formulate the oracle inequality for our original $d$-dimensional problem, 
we need to introduce the following concepts.
Firstly, for the decision function $\boldsymbol{f} : \mathbb R^d \to \mathbb R^d$,
it is necessary to introduce a loss function $\boldsymbol{L} : \mathbb R^d \to [0, \infty)$ such that
\begin{align*}
\boldsymbol{L} \left( X_i - \boldsymbol{f}(X_{i-1}) \right)
= \boldsymbol{L} \left( T^i(S_0) + \varepsilon_i - \boldsymbol{f}(T^{i-1}(S_0) + \varepsilon_{i-1}) \right)
\end{align*}
gives a value for the discrepancy between the forecast $\boldsymbol{f}(T^{i-1}(S_0) + \varepsilon_{i-1})$ and the
observation of the next state $T^i(S_0) + \varepsilon_i$. 
We say that a loss $\boldsymbol{L} : \mathbb R^d \rightarrow [0, \infty)$ can be \textit{clipped} at $M > 0$, if, for all 
$\boldsymbol{t} = (t_1, \ldots, t_d)  \in \mathbb R^d$, 
we have $\boldsymbol{L}(\wideparen{\boldsymbol{t}} \,) \leq \boldsymbol{L}(\boldsymbol{t})$, 
where $\wideparen{\boldsymbol{t}} = (\wideparen{t}_1, \ldots, \wideparen{t}_d)$ 
denotes the clipped value of $\boldsymbol{t}$ at $\{\pm M\}^d$. 
Moreover, the loss function $\boldsymbol{L} : \mathbb R^d \to [0, \infty)$ is called \textit{separable}, 
if there exists a distance-based loss 
$L : X \times Y \times \mathbb R \rightarrow [0, \infty)$ such that 
its representing function $\psi : \mathbb R \to [0, \infty)$ has a unique global minimum at $0$ and satisfies
\begin{align}
\boldsymbol{L}(\boldsymbol{r}) = \psi(r_1) + \cdots + \psi(r_d), 
\,\,\,\,\,\,\,\,
 \boldsymbol{r} = (r_1, \ldots, r_d) \in \mathbb R^d.
\label{dto1}
\end{align} 
In our problem-setting, the average forecasting performance is given by the $\boldsymbol{L}$-risk
\begin{align}
\mathcal R_{\boldsymbol{L}, \boldsymbol{P}} (\boldsymbol{f}) 
:= \iint \boldsymbol{L} \left( T(x) + \varepsilon_1 - \boldsymbol{f}(x + \varepsilon_0) \right) \, \nu(d \varepsilon) \, \mu(d x),
\label{Lrisk}
\end{align}
where $\varepsilon = (\varepsilon_i)_{i \geq 0}$ and $\boldsymbol{P} := \nu \otimes \mu$. 
Naturally, the smaller the risk, the better the forecaster is. 
Hence, we ideally would like to have a forecaster 
$\boldsymbol{f}_{\boldsymbol{L}, \boldsymbol{P}}^* : \mathbb R^d \rightarrow \mathbb R^d$
that attains the minimal $\boldsymbol{L}$-risk
\begin{align}
\mathcal R_{\boldsymbol{L}, \boldsymbol{P}}^* 
:= \inf \left\{ \mathcal R_{\boldsymbol{L}, \boldsymbol{P}}(\boldsymbol{f}) | 
\boldsymbol{f} : \mathbb R^d \to \mathbb R^d \text{ measurable} \right\}.
\label{MinimalLrisk}
\end{align}
The assumption (\ref{dto1}) then implies 
$\mathcal R_{\boldsymbol{L}, \boldsymbol{P}} (\boldsymbol{f}) 
= \sum_{j=1}^d \mathcal R_{L, P} (f_{{\boldsymbol{D}_n}^{(j)}})$
and 
$$\mathcal R_{\boldsymbol{L}, \mathbf{D}_n} (\boldsymbol{f}_{\!\boldsymbol{D}_n}) = 
\sum_{j=1}^d \mathcal R_{L, \mathbf{D}_n^{(j)}} (f_{\boldsymbol{D}_n^{(j)}})\, ,$$
where $\mathbf{D}_n$, $\mathbf{D}_n^{(j)}$ are the empirical measures associated to 
$\boldsymbol{D}_n$, $\boldsymbol{D}_n^{(j)}$ respectively.

Finally, let $\boldsymbol{L} : \mathbb R^d \rightarrow [0, \infty)$ be a clippable loss
and $\mathcal F$ be a hypothesis set with $0 \in \mathcal F$.  
A regularizer $\boldsymbol{\Upsilon}$ on $\ca F^d$,
that is, a function $\boldsymbol{\Upsilon} : \mathcal F^d \rightarrow [0, \infty)$,
is also said to be \textit{separable}, if there exists 
a regularizer $\Upsilon$ on $\ca F$ with $\Upsilon(0) = 0$ such that 
$\boldsymbol{\Upsilon} (\boldsymbol{f}) = \sum_{j=1}^d \Upsilon(f_j)$
for $\boldsymbol{f} = (f_1, \ldots, f_d)$.
Then, for $\delta \geq 0$, a learning method whose decision functions 
$\boldsymbol{f}_{\!\boldsymbol{D}_n, \boldsymbol{\Upsilon}} \in \mathcal F^d$ satisfy
\begin{align}
\boldsymbol{\Upsilon} ( \boldsymbol{f}_{\!\boldsymbol{D}_n, \boldsymbol{\Upsilon}} ) 
+ \mathcal R_{\boldsymbol{L}, \boldsymbol{D}_n} ( \wideparen{\boldsymbol{f}}_{\!\boldsymbol{D}_n, \boldsymbol{\Upsilon}}) 
 < \inf_{\boldsymbol{f} \in \mathcal F^d} \left( \boldsymbol{\Upsilon}(\boldsymbol{f}) 
  + \mathcal R_{\boldsymbol{L}, \boldsymbol{D}_n}(\boldsymbol{f}) \right) + d \delta 
\label{deltaCRERMd}
\end{align}
for all $n \geq 1$ and $\boldsymbol{D}_n \in (X \times Y)^{dn}$ is called 
$d \delta$-approximate clipped regularized empirical risk minimization 
($d \delta$-CR-ERM)
with respect to $\boldsymbol{L}$, $\mathcal F^d$, and $\boldsymbol{\Upsilon}$.

With all these preparations above, the oracle inequality for 
geometrically $\ca C$-mixing dynamical systems 
with i.i.d~noise processes, can be stated as following:

\begin{theorem} \label{oracleinequalityds}
Let $\Omega \subset \mathbb R^d$ be compact
and $(\Omega, \mathcal A, \mu, T)$ be a dynamical system.
Suppose that the stationary stochastic process $\ca T := (T^n)_{n \geq 0}$ 
is geometrically time-reversed $\ca C$-mixing 
and $\mathcal E = (\varepsilon_n)_{n \geq 0}$ is some i.i.d.~noise process 
defined on $(\Theta, \mathcal{C}, \nu)$ which is independent of $\mathcal T$.
Furthermore, let $\boldsymbol{L} : \mathbb R^d \to [0, \infty)$ be a clippable and separable loss function 
with the corresponding loss function $L : X \times Y \times \mathbb R \rightarrow [0, \infty)$ 
satisfying the properties described as in Theorem \ref{oracleinequality}.
Finally, let $\boldsymbol{\Upsilon} : \mathcal F^d \rightarrow [0, \infty)$ 
be a separable regularizer.
Then, for all fixed $\boldsymbol{f}_{\!0} = (f_0, \ldots, f_0)$, 
$\varepsilon > 0$, $\delta \geq 0$, $\tau \geq 1$, $n \geq n_0$ as in Theorem \ref{oracleinequality},
and $r \in (0, 1]$ satisfying (\ref{minradius}),
every learning method defined by (\ref{deltaCRERMd}) satisfies with probability 
$\mu \otimes \nu$ not less than $1 - 16 e^{-\tau}$:
\begin{align}
\boldsymbol{\Upsilon} ( \boldsymbol{f}_{\!\boldsymbol{D}_n, \boldsymbol{\Upsilon}} ) 
+ \mathcal R_{\boldsymbol{L}, \boldsymbol{P}} ( \wideparen{\boldsymbol{f}}_{\!\boldsymbol{D}_n, \boldsymbol{\Upsilon}}) 
- \mathcal R_{\boldsymbol{L}, \boldsymbol{P}}^*
 < 2 \boldsymbol{\Upsilon}(\boldsymbol{f}_{\!0}) 
  + 4 \mathcal R_{\boldsymbol{L}, \boldsymbol{P}}(\boldsymbol{f}_{\!0}) 
  - 4 \mathcal R_{\boldsymbol{L}, \boldsymbol{P}}^* 
  + 4 d r 
  + 5 d \varepsilon 
  + 2 d \delta.
\label{oracleinequalityd}
\end{align}
\end{theorem}

Again, this general oracle inequality can be applied to SVMs. We omit the details for the sake of brevity
and only mention that such applications would lead to learning rates and not only consistency as in \cite{StAn09a}.


\section{Proofs} \label{proofs}

\allowdisplaybreaks

\subsection{Proofs of Section \ref{StochasticProcesses}}

\begin{proof}[Proof of Example \ref{example-BV1}]
 Consider the collection $\Pi$ of ordered $n+1$-ples of points
$z_0 < z_1 < \ldots < z_n \in Z$,
where $n$ is an arbitrary natural number. The total variation
of a function $f : I \to \R$ is given by
\begin{align*}
 \|f\|_{BV(Z)} := \sup_{(z_0, z_1, \ldots, z_n) \in \Pi} \sum_{i=1}^n |f(z_i) - f(z_{i-1})|.
\end{align*}
Let us now assume that we have an $1 \leq i \leq n$ with $f(z_{i-1}) \leq f(z_i)$.
Moreover, for $t \leq 0$, it is not difficult to verify that $|1 - e^t| \leq |t|$.
This implies
\begin{align*}
\left| e^{f(z_i)} - e^{f(z_{i-1})} \right|
= e^{f(z_i)} \left| 1 - e^{f(z_{i-1}) - f(z_i)} \right|
\leq \bigl\| e^f \bigr\|_{\infty} |f(z_i) - f(z_{i-1})| \, .
\end{align*}
By interchanging the roles of $f(z_i)$ and $f(z_{i-1})$
we find the same estimate in the case of 
$f(z_{i-1}) \geq f(z_i)$.
Consequently we obtain
\begin{align*}
\sum_{i=1}^n |e^{f(z_i)} - e^{f(z_{i-1})|} 
\leq \bigl\| e^f \bigr\|_{\infty} \sum_{i=1}^n  |f(z_i) - f(z_{i-1})| 
\end{align*}
for all collection $\Pi$. 
Taking the supremum we get $\|e^f\|_{BV} \leq \| e^f \|_{\infty} \|f\|_{BV}$,
i.e. (\ref{expproperty}) is satisfied.
\end{proof}

\begin{proof}[Proof of Example \ref{example-Hoelder}]
Given a function $f \in C_{b,\alpha}(Z)$, we assume that 
$f(z) \geq f(z')$. Again, by using 
$|1 - e^t| \leq |t|$, $t \leq 0$, we obtain
\begin{align*}
\left| e^{f(z)} - e^{f(z')} \right|
= e^{f(z)} \left| 1 - e^{f(z') - f(z)} \right|
\leq \bigl\| e^f \bigr\|_{\infty} |f(z') - f(z)| 
\leq \bigl\| e^f \bigr\|_{\infty} |f|_{\a} |z - z'|^{\alpha}.
\end{align*}
By interchanging the roles of $f(z)$ and $f(z')$
we find the same estimate in the case of 
$f(z') \geq f(z)$.
Consequently we obtain
$\| e^f \| \leq  \| e^f \|_{\infty} |f|_{\a}$, 
i.e. (\ref{expproperty}) is satisfied.
\end{proof}

\begin{proof} [Proof of Theorem \ref{equivalencelambda}] 
($\Rightarrow$) The proof is straightforward. \\
($\Leftarrow$) 
For $p, q \in [1, \infty]$ with $1/p + 1/q = 1$, 
let $E_1$ and $E_2$ be Banach spaces that are continuously embedded into $L_p(\mu)$ and $L_q(\mu)$, respectively,
and let $F$ be a Banach space that is continuously embedded into $\ell_{\infty}$.
Analysis similar to that in the proof of \cite[Theorem 5.1]{StAn09a} shows
that if, for all $n \geq 0$, and all $h \in E_1$, $g \in E_2$, 
the correlation sequence satisfies
\begin{align*}
 \mathrm{cor}_{T, n}(h,g) \in F,
\end{align*}
then there exists a constant $c \in [0, \infty)$ such that
\begin{align} \label{corrinequality}
 \|\mathrm{cor}_{T, n}(h,g)\|_F
\leq c \cdot \|h\|_{E_1} \|g\|_{E_2},
\,\,\,\,\,\,\,\,
h \in E_1, \, g \in E_2.
\end{align}
In particular, (\ref{corrinequality}) holds for $E_1 = \ca C(\Om)$
and $E_2 = L_1(\mu)$ and the assertion is proved.
\end{proof}

\subsection{Proofs of Section \ref{bernstein}}

The following lemma, which may be of independent interest, supplies the key
to the proof of Theorem \ref{bernsteininequality}.

\begin{lemma} \label{lemma2}
Let $\mathcal Z := (Z_n)_{n \geq 0}$ be a $Z$-valued stationary 
(time-reversed) $\ca C$-mixing process
on the probability space $(\Omega, \mathcal A, \mu)$ with rate $(d_n)_{n \geq 0}$,
and $P := \mu_{Z_0}$. 
Moreover, for $f : Z \to [0, \infty)$,
suppose that $f \in \ca C(Z)$
and write 
$f_n := f \circ Z_n$. 
Finally, 
assume that we have natural numbers $k$ and $l$ satisfying 
\begin{align}
2 l \cdot \|f\|_{\ca C} \cdot d_k \leq \|f\|_{L_1(P)}.
\label{assumption}
\end{align}
Then we have
\begin{align*}
\mathbb E_{\mu} \prod_{j=0}^l f_{jk} \leq 2 \|f\|_{L_1(P)}^{l+1}.
\end{align*} 
\end{lemma}

\begin{proof}[Proof of Lemma \ref{lemma2}]
We divide the proof into two parts. \\ 
\textit{(i)} Suppose that the correlation inequality (\ref{decayd}) holds.
Obviously the case $f = 0$ $P$-a.s.~is trivial. For $f \neq 0$, we define
\begin{align} \label{dldef}
 D_l := \left| \mathbb E_{\mu} \prod_{j=0}^l f_{jk} 
                 - \prod_{j=0}^l \mathbb E_{\mu} f_{jk} \right|.
\end{align}
Then we have
\begin{align*}
 D_l
& \leq \left| \mathbb E_{\mu} \left( \prod_{j=0}^{l-1} f_{jk} \right) \, f_{lk} 
              - \mathbb E_{\mu} \prod_{j=0}^{l-1} f_{jk} \, \mathbb E_{\mu} f_{lk} \right| 
       + \left| \mathbb E_{\mu} \prod_{j=0}^{l-1} f_{jk} \, \mathbb E_{\mu} f_{lk}
                              - \prod_{j=0}^l \mathbb E_{\mu} f_{jk} \right| 
\\
& = \left| \mathbb E_{\mu} \left( \prod_{j=0}^{l-1} f_{jk} \right) \, f_{lk} 
              - \mathbb E_{\mu} \prod_{j=0}^{l-1} f_{jk} \, \mathbb E_{\mu} f_{lk} \right| 
       + \left| \mathbb E_{\mu} \prod_{j=0}^{l-1} f_{jk} \, \mathbb E_{\mu} f_{lk}
                - \prod_{j=0}^{l-1} \mathbb E_{\mu} f_{jk} \, \mathbb E_{\mu} f_{lk} \right|.
\end{align*}
Since the stochastic process $\mathcal Z$ is stationary, 
the decay of correlations (\ref{decayd}) together with 
$\psi:=\prod_{j=0}^{l-1} f_{jk}$, $h:=f$, and
the assumption $f \geq 0$ yields
\begin{align*}
 \Bigg| \mathbb E_{\mu} \left( \prod_{j=0}^{l-1} f_{jk} \right) \, f_{lk} 
            - \mathbb E_{\mu} \prod_{j=0}^{l-1} f_{jk} \, \mathbb E_{\mu} f_{lk} \Bigg|
& \leq \Big\| \prod_{j=0}^{l-1} f_{jk} \Big\|_{L_1(\mu)} \|f\|_{\ca C} \, d_k
= \Big| \mathbb E_{\mu} \prod_{j=0}^{l-1} f_{jk} \Big| \|f\|_{\ca C} \, d_k
\\
& \leq \left( \left| \mathbb E_{\mu} \prod_{j=0}^{l-1} f_{jk} - \prod_{j=0}^{l-1} \mathbb E_{\mu} f_{jk} \right|
+ \prod_{j=0}^{l-1} \mathbb E_{\mu} f_{jk} \right)  \|f\|_{\ca C} \, d_k
\\
& = \left( D_{l-1} + \|f\|_{L_1(P)}^l \right)  \|f\|_{\ca C} \, d_k.
\end{align*}
Moreover, for the second term, we find
\begin{align*}
\left| \mathbb E_{\mu} \prod_{j=0}^{l-1} f_{jk} \, \mathbb E_{\mu} f_{lk}
                - \prod_{j=0}^{l-1} \mathbb E_{\mu} f_{jk} \, \mathbb E_{\mu} f_{lk} \right|
=  \|f\|_{L_1(P)}  \left| \mathbb E_{\mu} \prod_{j=0}^{l-1} f_{jk} - \prod_{j=0}^{l-1} \mathbb E_{\mu} f_{jk} \right| 
= \|f\|_{L_1(P)} D_{l-1} .
\end{align*}
These estimates together imply that 
\begin{align}
  D_l 
& \leq \left( D_{l-1} + \|f\|_{L_1(P)}^l \right)  \|f\|_{\ca C} \, d_k 
        + \|f\|_{L_1(P)}  D_{l-1}
\nonumber\\
& = \left( \|f\|_{L_1(P)} + \|f\|_{\ca C} \,  d_k \right)  D_{l-1} 
        + \|f\|_{\ca C} \|f\|_{L_1(P)}^l \, d_k.
\label{dlstar}
\end{align}
In the following, we will show by induction that 
the latter estimate implies
\begin{align}
D_l \leq \|f\|_{L_1(P)} \left( \left( \|f\|_{L_1(P)} + \|f\|_{\ca C} \,  d_k \right)^l - \|f\|_{L_1(P)}^l \right).
\label{dl}
\end{align}
When $l = 1$, (\ref{dl}) is true because of (\ref{decayd}). 
Now let $l \geq 1$ be given and suppose (\ref{dl}) is true for $l$. Then
(\ref{dlstar}) and (\ref{dl}) imply 
\begin{align*}
 D_{l+1}
& \leq \left( \|f\|_{L_1(P)} + \|f\|_{\ca C} \, d_k \right) D_l + \|f\|_{\ca C} \|f\|_{L_1(P)}^{l+1} \, d_k
\\
& \leq \left( \|f\|_{L_1(P)} + \|f\|_{\ca C} \, d_k \right)
\left(
\|f\|_{L_1(P)} \left( \left( \|f\|_{L_1(P)} + \|f\|_{\ca C} \, d_k \right)^l - \|f\|_{L_1(P)}^l \right)
\right)
+ \|f\|_{\ca C} \|f\|_{L_1(P)}^{l+1} \, d_k
\\
& = \|f\|_{L_1(P)} \left( \left( \|f\|_{L_1(P)} + \|f\|_{\ca C} \,  d_k \right)^{l+1} - \|f\|_{L_1(P)}^{l+1} \right).
\end{align*}
Thus, (\ref{dl}) holds for $l + 1$, and the proof of the induction step is complete.
By the principle of induction, (\ref{dl}) is thus true for all $l \geq 1$.

Using the binomial formula, we obtain 
\begin{align*}
D_l \leq \|f\|_{L_1(P)} \left( 
\sum_{i=0}^l \binom{l}{i} \|f\|_{L_1(P)}^{l-i} \left( \|f\|_{\ca C} \,  d_k \right)^i - \|f\|_{L_1(P)}^l \right).
\end{align*}
For $i = 0, \ldots, l$ we now set
\begin{align*}
 a_i := \binom{l}{i} \|f\|_{L_1(P)}^{l-i} \left( \|f\|_{\ca C} \,  d_k \right)^i.
\end{align*}
The assumption (\ref{assumption}) implies for $i = 0, \ldots, l-1$
\begin{align*}
 \frac{a_{i+1}}{a_i}
& = \frac{\binom{l}{i+1} \|f\|_{L_1(P)}^{l-i-1} \left( \|f\|_{\ca C} \, d_k \right)^{i+1}}
         {\binom{l}{i} \|f\|_{L_1(P)}^{l-i} \left( \|f\|_{\ca C} \, d_k \right)^i}
  = \frac{\frac{l!}{(i+1)!(l-i-1)!}}{\frac{l!}{i!(l-i)!}} \frac{\|f\|_{\ca C} \, d_k}{\|f\|_{L_1(P)}}
\\
& = \frac{l-i}{i+1} \frac{\|f\|_{\ca C} \, d_k}{\|f\|_{L_1(P)}}
  \leq l \cdot \frac{\|f\|_{\ca C}}{\|f\|_{L_1(P)}} \cdot d_k 
  \leq \frac{1}{2}.
\end{align*}
This gives $a_i \leq 2^{-i} a_0$ for all $i = 0, \ldots, l$ 
and consequently we have
\begin{align*}
 \sum_{i=0}^l a_i
= a_0 + \sum_{i=1}^l a_i
\leq a_0 + \sum_{i=1}^l 2^{-i} a_0
= a_0 \cdot \left( \sum_{i=1}^l 2^{-i} \right)
\leq 2 a_0.
\end{align*}
This implies
\begin{align*}
  D_l 
& \leq \|f\|_{L_1(P)} \left( \sum_{i=0}^l a_i - \|f\|_{L_1(P)}^l \right)
  \leq \|f\|_{L_1(P)} \left( 2 a_0 - \|f\|_{L_1(P)}^l \right)
\\
& = \|f\|_{L_1(P)} \left( 2 \|f\|_{L_1(P)}^l - \|f\|_{L_1(P)}^l \right)
  = \|f\|_{L_1(P)}^{l+1}.
\end{align*}
Using the definition of $D_l$ we thus obtain
\begin{align*}
\mathbb E_{\mu} \prod_{j=0}^l f_{jk} \leq 2 \|f\|_{L_1(P)}^{l+1}.
\end{align*} 
\textit{(ii)} Suppose that the correlation inequality (\ref{decaydrev}) holds. \\
Again, the case $f = 0$ $P$-a.s.~is trivial. 
For $f \neq 0$, we estimate $D_l$ defined as in (\ref{dldef}) 
in a slightly different way from above: 
\begin{align*}
 D_l
& \leq \left| \mathbb E_{\mu} f_0 \prod_{j=1}^l f_{jk} 
                 - \mathbb E_{\mu} f_0 \mathbb E_{\mu} \prod_{j=1}^l f_{jk} \right| 
       + \left| \mathbb E_{\mu} f_0 \mathbb E_{\mu} \prod_{j=1}^l f_{jk} 
                              - \prod_{j=0}^l \mathbb E_{\mu} f_{jk} \right| 
\\
& = \left| \mathbb E_{\mu} f_0 \prod_{j=1}^l f_{jk} 
              - \mathbb E_{\mu} f_0 \mathbb E_{\mu} \prod_{j=1}^l f_{jk} \right| 
       + \left| \mathbb E_{\mu} f_0 \mathbb E_{\mu} \prod_{j=1}^l f_{jk} 
                 - \mathbb E_{\mu} f_0 \prod_{j=1}^l \mathbb E_{\mu} f_{jk} \right|. 
\end{align*}
Since the stochastic process $\mathcal Z$ is stationary, 
the decay of correlations (\ref{decaydrev}) together with 
$h:=f$, $\phi:=\prod_{j=1}^l f_{jk}$, and
the assumption $f \geq 0$ yields
\begin{align*}
 \left| \mathbb E_{\mu} f_0 \prod_{j=1}^l f_{jk} 
            - \mathbb E_{\mu} f_0 \mathbb E_{\mu} \prod_{j=1}^l f_{jk} \right| 
& \leq \|f\|_{\ca C} \Big\| \prod_{j=1}^l f_{jk} \Big\|_{L_1(\mu)}  \, d_k
\\
& = \|f\|_{\ca C} \Big| \mathbb E_{\mu} \prod_{j=1}^l f_{jk} \Big|  \, d_k
  = \|f\|_{\ca C} \Big| \mathbb E_{\mu} \prod_{j=0}^{l-1} f_{jk} \Big|  \, d_k
\\
& \leq \|f\|_{\ca C} \left( \left| \mathbb E_{\mu} \prod_{j=0}^{l-1} f_{jk} - \prod_{j=0}^{l-1} \mathbb E_{\mu} f_{jk} \right| 
+ \prod_{j=0}^{l-1} \mathbb E_{\mu} f_{jk} \right)  d_k
\\
& = \|f\|_{\ca C} \left( D_{l-1} + \|f\|_{L_1(P)}^l \right)  d_k.
\end{align*}
Moreover, for the second term, since the stochastic process $\mathcal Z$ is stationary, we find
\begin{align*}
\left| \mathbb E_{\mu} f_0 \mathbb E_{\mu} \prod_{j=1}^l f_{jk} 
         - \mathbb E_{\mu} f_0 \prod_{j=1}^l \mathbb E_{\mu} f_{jk} \right|
& = \|f\|_{L_1(P)}  \left| \mathbb E_{\mu} \prod_{j=1}^l f_{jk} 
                                  - \prod_{j=1}^l \mathbb E_{\mu} f_{jk} \right| 
\\
& = \|f\|_{L_1(P)}  \left| \mathbb E_{\mu} \prod_{j=0}^{l-1} f_{jk} 
                                  - \prod_{j=0}^{l-1} \mathbb E_{\mu} f_{jk} \right|
\\
& = \|f\|_{L_1(P)}  D_{l-1}.
\end{align*}
Combining the above estimates, we get 
\begin{align*}
  D_l 
& \leq \|f\|_{\ca C} \left( D_{l-1} + \|f\|_{L_1(P)}^l \right)  d_k 
        + \|f\|_{L_1(P)}  D_{l-1}
\nonumber\\
& = \left( \|f\|_{L_1(P)} + \|f\|_{\ca C} \,  d_k \right)  D_{l-1} 
        + \|f\|_{\ca C} \|f\|_{L_1(P)}^l \, d_k.
\end{align*}
This estimate coincides with (\ref{dlstar}).
The rest of the argument is the same as in \textit{(i)},
and the assertion is proved.
\end{proof}

To prove Theorem \ref{bernsteininequality}, we need to introduce some notations. 
In the following, for $t \in \mathbb R$, $\lfloor t \rfloor$ is the largest integer $n$ satisfying $n \leq t$, 
and similarly, $\lceil t \rceil$ is the smallest integer $n$ satisfying $n \geq t$.  
We write $h_i := h \circ Z_i$ and
\begin{align*}
S_n = \sum_{i=1}^n h_i = \sum_{i=1}^n h \circ Z_i.
\end{align*}
We now recall the so-called blocking method. 
To this end, we partition the set $\{ 1, 2, \ldots, n \}$ into $k$ blocks. 
Each block will contain approximatively $l := \lfloor n / k \rfloor$ terms. 
Let $r := n - k \cdot l < k$ denote the remainder when we divide $n$ by $k$. 

We now construct $k$ blocks as follows. Define $I_i$, the indexes of terms in the $i$-th block, as
\begin{align*}
I_i = 
\begin{cases}
\{ i, i + k, \ldots, i + (l+1) k \}, &~~ \text{if } 1 \leq i \leq r, \\
\{ i, i + k, \ldots, i + l k \}, &~~ \text{if } r + 1 \leq i \leq k. \\
\end{cases}
\end{align*}
Note that the number of the terms satisfies
\begin{align*}
|I_i| = 
\begin{cases}
l + 1, &~~ \text{for } 1 \leq i \leq r, \\
l, &~~ \text{for } r + 1 \leq i \leq k. \\
\end{cases}
\end{align*}
In other words, the first $r$ blocks each contain $l + 1$ terms, while the last $(k - r)$ blocks each contain $l$ terms.
Moreover, we have
\begin{align}
\sum_{i=1}^k |I_i| = \sum_{i=1}^r |I_i| + \sum_{i=r+1}^k |I_i| = r (l + 1) + (k - r) l = n.
\label{block}
\end{align}
Furthermore, for $i = 1, 2, \ldots, k$, we define the $i$-th block sum as
\begin{align}
g_i = \sum_{j \in I_i} h_j
\label{ithblock}
\end{align}
such that
\begin{align}
S_n = \sum_{i=1}^k g_i.
\label{sandg}
\end{align}
Finally, for $i = 1, 2, \ldots, k$, define 
\begin{align}
p_i := \frac{|I_i|}{n}.
\label{pi} 
\end{align}
It follows from (\ref{block}) that
\begin{align*}
\sum_{i=1}^k p_i = \frac{1}{n} \sum_{i=1}^k |I_i| = 1.
\end{align*}

The following three lemmas will derive 
the upper bounds for the expected value of the exponentials of $S_n$.

\begin{lemma}  \label{lemma3}
Let $\mathcal Z := (Z_n)_{n \geq 0}$ be a $Z$-valued stationary stochastic process 
on the probability space $(\Omega, \mathcal A, \mu)$ 
and $P := \mu_{Z_0}$. 
Moreover, let $k$ and $l$ be defined as above,
and for a bounded $h : Z \to \mathbb R$ we define 
$g_i$ and $S_n$ by (\ref{ithblock}) and (\ref{sandg}), respectively.
Then, for all $t > 0$, we have
\begin{align*}
\mathbb E_{\mu} \exp \left( t \frac{S_n}{n} \right) 
\leq \sum_{i=1}^k p_i \mathbb E_{\mu} \exp \left( t \frac{g_i}{|I_i|} \right).
\end{align*}
\end{lemma}

\begin{proof}[Proof of Lemma \ref{lemma3}]
It is well-known that the exponential function is convex.
Jensen's inequality together with $\sum_{i=1}^k p_i = 1$, (\ref{sandg}), and (\ref{pi}) yields 
\begin{align*}
\mathbb E_{\mu} \exp \left( t \frac{S_n}{n} \right) 
= \mathbb E_{\mu} \exp \left( \sum_{i=1}^k t p_i \frac{g_i}{|I_i|} \right) 
\leq \sum_{i=1}^k p_i \mathbb E_{\mu} \exp \left( t \frac{g_i}{|I_i|} \right).
\end{align*}
\end{proof}

\begin{lemma} \label{lemma4}
Let $\mathcal Z := (Z_n)_{n \geq 0}$ be a $Z$-valued stationary (time-reversed) $\ca C$-mixing process
on the probability space $(\Omega, \mathcal A, \mu)$ with rate $(d_n)_{n \geq 0}$,
and $P := \mu_{Z_0}$. 
Moreover, for $h : Z \to [0, \infty)$, we write 
$h_n := h \circ Z_n$.
Finally, let $k$ and $l$ be defined as above.
Then, for all $t > 0$ satisfying
\begin{align}
e^{\frac{t}{|I_i|} h} \in \ca C(Z) 
\,\, \text{ and } \, \,
2 l \cdot 
\|e^{\frac{t}{|I_i|} h}\|_{\ca C}
\cdot d_k 
\leq \|e^{\frac{t}{|I_i|} h}\|_{L_1(P)},
\label{assumption2}
\end{align}
we have
\begin{align*}
\mathbb E_{\mu} \exp \left( t \frac{g_i}{|I_i|} \right)  
\leq 2 \left( \mathbb E_P \exp \left( t \frac{h}{|I_i|} \right) \right)^{|I_i|}.
\end{align*}
\end{lemma}

\begin{proof}[Proof of Lemma \ref{lemma4}]
The $i$th block sum $g_i$ in (\ref{ithblock}) depends only on $h_{i+jk}$ with $j$ ranging from $0$ through $|I_i| - 1$. 
Since $\mathcal Z$ is stationary, Lemma \ref{lemma2} with $f:=\exp(\frac{t}{|I_i|} h)$ then yields  
\begin{align*}
\mathbb E_{\mu} \exp \left( t \frac{g_i}{|I_i|} \right) 
& = \mathbb E_{\mu} \exp \left( \frac{t}{|I_i|} \sum_{j=0}^{|I_i|-1} h_{i+jk} \right)  
  = \mathbb E_{\mu} \exp \left( \frac{t}{|I_i|} \sum_{j=0}^{|I_i|-1} h_{jk} \right)  
\\
& = \mathbb E_{\mu} \prod_{j=0}^{|I_i|-1} \exp \left( \frac{t}{|I_i|} h_{jk} \right)  
  \leq 2 \left( \mathbb E_P \exp \left( t \frac{h}{|I_i|} \right) \right)^{|I_i|}.
\end{align*}
\end{proof}

\begin{lemma}   \label{lemma5}
Let $\mathcal Z := (Z_n)_{n \geq 0}$ be a $Z$-valued stationary (time-reversed) $\ca C$-mixing process 
on the probability space $(\Omega, \mathcal A, \mu)$ with rate $(d_n)_{n \geq 0}$,
and $P := \mu_{Z_0}$. 
Moreover, for $h : Z \to [0, \infty)$, we write 
$h_n := h \circ Z_n$ and suppose that $\mathbb E_P h = 0$,
$\|h\| \leq A$, $\|h\|_{\infty} \leq B$, and $\mathbb E_P h^2 \leq \sigma^2$ 
for some $A > 0$, $B > 0$ and $\sigma \geq 0$.
Finally, let $k$ and $l$ be defined as above.
Then, for all $i = 1, \ldots, k$, and all $t > 0$ satisfying
$0 < t < 3 l/ B$ and (\ref{assumption2}),
we have
\begin{align*} 
\mathbb E_{\mu} \exp \left( t \frac{g_i}{|I_i|} \right)  
\leq 2 \exp \left( \frac{t^2 \sigma^2}{2 (l - t B / 3)} \right).
\end{align*}
\end{lemma}

\begin{proof}[Proof of Lemma \ref{lemma5}]
Because of $\|h\|_{\infty} \leq B$ and $2 \cdot 3^{j-2} \leq j!$, we obtain
\begin{align*}
\exp \left( \frac{t}{|I_i|} h \right) 
& = 1 + \frac{t}{|I_i|} h 
      + \sum_{j=2}^{\infty} \left( \frac{t}{|I_i|} \right)^j \frac{h^j}{j!}
\\ 
& \leq 1 + \frac{t}{|I_i|} h 
         + \sum_{j=2}^{\infty} \left( \frac{t}{|I_i|} \right)^j \frac{h^2 B^{j-2}}{2 \cdot 3^{j-2}}
\\ 
& = 1 + \frac{t}{|I_i|} h 
      + \frac{1}{2} \left( \frac{t}{|I_i|} \right)^2 h^2 \sum_{j=2}^{\infty} \left( \frac{t B}{3 |I_i|} \right)^{j-2}
\\ 
& = 1 + \frac{t}{|I_i|} h 
      + \frac{1}{2} \left( \frac{t}{|I_i|} \right)^2 h^2 \frac{1}{1 - t B / (3 |I_i|)}
\end{align*}
if $t B / (3 |I_i|) < 1$. This, together with $\mathbb E_P h = 0$, $1 + x \leq e^x$, and 
$l \leq |I_i| \leq l +1$, implies
\begin{align}
\left( \mathbb E_P \exp \left( t \frac{h}{|I_i|} \right) \right)^{|I_i|} 
& \leq \left( 1 + \frac{1}{2} \left( \frac{t}{|I_i|} \right)^2 \sigma^2 \frac{1}{1 - t B / (3 |I_i|)} \right)^{|I_i|} 
\nonumber\\
& \leq \left( \exp \left( \frac{1}{2} \left( \frac{t}{|I_i|} \right)^2 \sigma^2 \frac{1}{1 - t B / (3 |I_i|)} \right) \right)^{|I_i|} 
\nonumber\\
& = \exp \left( \frac{t^2 \sigma^2}{2 (|I_i| - t B / 3)} \right) 
\nonumber\\ 
& \leq \exp \left( \frac{t^2 \sigma^2}{2 (l - t B / 3)} \right), 
\label{term1}
\end{align}
since the assumed 
$t B / (3 l) < 1$
implies $t B / (3 |I_i|) < 1$. Lemma \ref{lemma4} then yields 
\begin{align*} 
\mathbb E_{\mu} \exp \left( t \frac{g_i}{|I_i|} \right)  
\leq 2 \exp \left( \frac{t^2 \sigma^2}{2 (l - t B / 3)} \right).
\end{align*}
\end{proof}

\begin{proof}[Proof of Theorem \ref{bernsteininequality}]
For $k$ and $l$ as above we define
\begin{align}
t := \frac{l \varepsilon}{\sigma^2 + \varepsilon B / 3}.
\label{valuet}
\end{align}
Then we have
\begin{align}
\frac{t}{|I_i|} 
\leq \frac{t}{l} 
= \frac{\varepsilon}{\sigma^2 + \varepsilon B / 3}
\leq \frac{\varepsilon}{\varepsilon B / 3} = \frac{3}{B}.
\label{tfirstproperty}
\end{align}
In particular, this $t$ satisfies $0 < t < 3 l/ B$. Moreover, we find
\begin{align}
 \left\| \exp \left( \frac{t}{|I_i|} h \right) \right\|_{\infty} 
\leq \exp \left( \frac{3}{B} \cdot B \right) 
= e^3.
\label{infinitybd}
\end{align}
Then, the assumption (\ref{expproperty})
together with the bounds (\ref{infinitybd}) and (\ref{tfirstproperty}) implies
\begin{align}
 \left\| \exp \left( \frac{t}{|I_i|} h \right) \right\|
\leq \left\| \exp \left( \frac{t}{|I_i|} h \right) \right\|_{\infty} \left\| \frac{t}{|I_i|} h \right\|
\leq e^3 \cdot \frac{t}{|I_i|} \|h\|
\leq \frac{3 e^3 A}{B}.
\label{lipbd}
\end{align}
Since $- B \leq h \leq B$, we further find
\begin{align}
 \left\| \exp \left( \frac{t}{|I_i|} h \right) \right\|_{L_1(P)}
=  \mathbb E_P \exp \left( \frac{t}{|I_i|} h \right)
\geq \exp \left( \frac{3}{B} \cdot (-B) \right)
= e^{-3}.
\label{l1bd}
\end{align}
Now we choose 
$k := \lfloor (\log n)^{\frac{2}{\gamma}} \rfloor + 1$, which implies
$k \geq (\log n)^{\frac{2}{\gamma}}$.
On the other hand, since $(\log n)^{\frac{2}{\gamma}} \geq 1$ for $n \geq n_0 \geq 3$,
we have $k \leq 2 (\log n)^{\frac{2}{\gamma}}$. This implies
\begin{align}
 l = \frac{n-r}{k} \geq \frac{n}{k} - 1
 \geq \frac{1}{2} \frac{n}{(\log n)^{\frac{2}{\gamma}}} - 1
 \geq \frac{1}{4} \frac{n}{(\log n)^{\frac{2}{\gamma}}},
\label{lowerl}
\end{align}
since we have $n \geq 4 (\log n)^{\frac{2}{\gamma}}$ for $n \geq n_0$. 
Now, by (\ref{infinitybd}), (\ref{lipbd}), (\ref{l1bd}), (\ref{dn}),
and (\ref{nzero}) we obtain
\begin{align*}
 l \cdot \frac{\| e^{\frac{t}{|I_i|} h} \|_{\ca C}}{\|e^{\frac{t}{|I_i|} h}\|_{L_1(P)}} \cdot d_k 
& \leq l \cdot \frac{\|e^{\frac{t}{|I_i|} h}\|_{\infty} + \|e^{\frac{t}{|I_i|} h}\|}{\|e^{\frac{t}{|I_i|} h}\|_{L_1(P)}} 
             \cdot c \cdot \exp \left( - b k^{\gamma} \right)
\\
& \leq n \cdot \frac{e^3 + \frac{3 e^3 A}{B}}{e^{-3}} \cdot c \cdot \exp \left( - b (\log n)^2 \right)
\\
& \leq n \cdot \frac{404 c (3 A + B)}{B} \cdot \exp \left( - b \log n \cdot \frac{3}{b} \right)
\\
& \leq n \cdot \frac{n^2}{2} \cdot n^{-3} = \frac{1}{2},
\end{align*}
i.e., the assumption (\ref{assumption2}) is valid. 

Summarizing, the value of $t$ defined as in (\ref{valuet}) satisfies
$0 < t < 3 l/ B$ 
and the assumption (\ref{assumption2}). In other words,
all the requirements on $t$ in Lemma \ref{lemma5} are satisfied.

Now, for this $t$, 
by using Markov's inequality, Lemma \ref{lemma3}, and Lemma \ref{lemma5}, we obtain
for any $\varepsilon > 0$,
\begin{align}
 P \left( \frac{S_n}{n} > \varepsilon \right)
& = P \left( \exp \left( t \frac{S_n}{n} \right) > \exp \left( t \varepsilon \right) \right) 
\nonumber\\
& \leq \exp \left( - t \varepsilon \right) \mathbb E_{\mu} \exp \left( t \frac{S_n}{n} \right) 
\nonumber\\
& \leq \exp \left( - t \varepsilon \right) 
       \sum_{i=1}^k p_i \mathbb E_{\mu} \exp \left( t \frac{g_i}{|I_i|} \right)
\nonumber\\
& \leq \exp \left( - t \varepsilon \right) \cdot
2 \exp \left( \frac{t^2 \sigma^2}{2 (l - t B / 3)} \right)
\sum_{i=1}^k p_i
\nonumber\\
& = 2 \exp \left( - t \varepsilon + \frac{t^2 \sigma^2}{2 (l - t B / 3)} \right).
\label{zwischen}
\end{align}
Substituting the definition of $t$ into the exponent of inequality (\ref{zwischen}), we get 
\begin{align*}
- t \varepsilon + \frac{t^2 \sigma^2}{2 (l - t B / 3)} 
& = - \frac{l \varepsilon^2}{\sigma^2 + \varepsilon B / 3} 
+ \frac{l^2 \varepsilon^2}{\left( \sigma^2 + \varepsilon B / 3 \right)^2}
\cdot \frac{\sigma^2}{2 \left( l - \frac{l \varepsilon B / 3}{\sigma^2 + \varepsilon B / 3} \right)}
\nonumber\\
& = - \frac{l \varepsilon^2}{\sigma^2 + \varepsilon B / 3} 
+ \frac{l \varepsilon^2}{\sigma^2 + \varepsilon B / 3}
\cdot \frac{\sigma^2}{2 \left( \sigma^2 + \varepsilon B / 3 - \varepsilon B / 3 \right)}
\nonumber\\
& = \frac{- l \varepsilon^2}{2 \left( \sigma^2 + \varepsilon B / 3 \right)},
\nonumber
\end{align*}
hence
\begin{align*}
\mathbb P \left( \frac{1}{n} S_n > \varepsilon \right) 
\leq 2 \exp \left( - \frac{- l \varepsilon^2}{2 \left( \sigma^2 + \varepsilon B / 3 \right)} \right).
\end{align*}
Using the estimate (\ref{lowerl}), we thus obtain
\begin{align*}
\mathbb P \left( \frac{1}{n} S_n > \varepsilon \right) 
\leq 2 \exp \left( - \frac{n \varepsilon^2}{8 (\log n)^{\frac{2}{\gamma}} \left( \sigma^2 + \varepsilon B / 3 \right)} \right),
\end{align*}
for all $n \geq n_0$ and $\varepsilon > 0$.
Setting $\tau := \frac{n \varepsilon^2}{8 (\log n)^{\frac{2}{\gamma}} \left( \sigma^2 + \varepsilon B / 3 \right)}$, we then have
\begin{align*}
\mu \left( \left\{ \omega \in \Omega : 
\frac{1}{n} \sum_{i=1}^n h(Z_i(\omega)) \geq \varepsilon \right\} \right) \leq 2 e^{- \tau}, ~~~~~~ 
n \geq n_0.
\end{align*}
Simple transformations and estimations then yield
\begin{align*}
\mu \left( \left\{ \omega \in \Omega : 
\frac{1}{n} \sum_{i=1}^n h(Z_i(\omega)) 
\geq \sqrt{\frac{8 (\log n)^{\frac{2}{\gamma}} \tau \sigma^2}{n}} 
     + \frac{8 (\log n)^{\frac{2}{\gamma}} B \tau}{3 n} \right\} \right) 
\leq 2 e^{- \tau}  
\end{align*}
for all $n \geq n_0$ and $\tau > 0$.
\end{proof}

\subsection{Proofs of Section \ref{applications}}

\begin{proof}[Proof of Lemma \ref{lipschitzlemma}] 
\textit{(i)}
For the least square loss (\ref{lsloss}), by using $a + b \leq (2(a^2 + b^2))^{1/2}$, we obtain
\begin{align*}
|L(x, y, f(x)) - L(x', y', f(x'))|
& = |(y - f(x))^2 - (y' - f(x'))^2|
\\
& = |y - f(x) + y' - f(x')| \cdot |y - f(x) - y' + f(x')|
\\
& \leq \left( |y + y'| + |f(x) + f(x')| \right) \left( |y - y'| + |f(x) - f(x')| \right)
\\
& \leq 2 \left( M + \|f\|_{\infty} \right) \left( |y - y'| + |f|_1 |x - x'| \right)
\\
& \leq 2 \left( M + \|f\|_{\infty} \right) (1+|f|_1) \left( |y - y'| +  |x - x'| \right)
\\
& \leq 2\sqrt{2} \left( M + \|f\|_{\infty} \right) (1+|f|_1) \|(x, y) - (x', y')\|_2 
\end{align*}
for all $(x, y), (x', y') \in X \times Y$, that is, we have proved the assertion. 

\textit{(ii)}
Let $L$ be the 
the $\t$-pinball loss (\ref{pbloss}) and define
\begin{align*}
 D := L(x, y, f(x)) - L(x', y', f(x')).
\end{align*}
We divide the proof into the following four cases.
If $y \geq f(x)$ and $y' \geq f(x')$, we have
\begin{align*}
 |D| 
= |\tau(y - f(x)) - \tau(y' - f(x'))|
= \tau|(y - y') - (f(x) - f(x'))|. 
\end{align*}
If $y < f(x)$ and $y' < f(x')$, in an exactly similar way we obtain
\begin{align*}
 |D| = (1-\tau) |(y - y') - (f(x) - f(x'))|.
\end{align*}
Moreover, in case of $y \geq f(x)$ and $y' < f(x')$, we get  
\begin{align*}
 |D| 
= |\tau(y - f(x)) + (1-\tau)(y' - f(x'))|
\leq |(y - f(x)) + (f(x') - y')|.
\end{align*}
Similar arguments to the case $y < f(x)$ and $y' \geq f(x')$ show that
\begin{align*}
 |D| 
= |-(1-\tau)(y - f(x)) - \tau(y' - f(x'))|
\leq |(y - f(x)) + (f(x') - y')|.
\end{align*}
Summarizing, for all $(x, y), (x', y') \in X \times Y$, we have
\begin{align*}
 |L(x, y, f(x)) - L(x', y', f(x'))|
\leq |(y - y') - (f(x) - f(x'))|
\leq |y - y'| + |f(x) - f(x')|.
\end{align*}
The rest of the argument is similar to that of part \textit{(i)},
and the assertion is proved.
\end{proof}

%
%

For our proof of Theorem \ref{oracleinequality} we need 
the following simple and well-known lemma (see e.g. \cite[Lemma 7.1]{StCh08a}): 
\begin{lemma} \label{younginequality}
For $q \in (1, \infty)$, define $q' \in (1, \infty)$ by $1/q + 1/q' = 1$. Then, for all $a, b \geq 0$, we have 
$(q a)^{2/q} (q' b)^{2/q'} \leq (a + b)^2$ and $a b \leq a^q / q + b^{q'}/q'$.
\end{lemma}

Apart from 
the semi-norm bounds involving $A_0$, $A_1$, and $A^*$ and
some constants, for example, 
the constant $n_0$ and
the constants on the right side of the oracle inequality, 
the proof of Theorem \ref{oracleinequality} is almost identical to 
the proof of \cite[Theorem 3.1]{HaSt14a}.
For this reason, a few parts of the proof will be omitted.

\begin{proof}[Proof of Theorem \ref{oracleinequality}]
\textbf{Main Decomposition.}
For $f : X \rightarrow \mathbb R$ we define $h_f := L \circ f - L \circ f_{L,P}^*$. 
By the definition of $f_{D_n, \Upsilon}$, we then have 
\begin{align*}
\Upsilon(f_{D_n, \Upsilon}) + \mathbb E_{D_n} h_{\wideparen{f}_{D_n, \Upsilon}} 
\leq \Upsilon (f_0) + \mathbb E_{D_n} h_{f_0} + \delta,
\end{align*} 
and consequently we obtain
\begin{align}
& \Upsilon(f_{D_n, \Upsilon}) + \mathcal R_{L, P} (\wideparen{f}_{D_n, \Upsilon}) - \mathcal R_{L, P}^*
\nonumber\\
& = \Upsilon(f_{D_n, \Upsilon}) + \mathbb E_P h_{\wideparen{f}_{D_n, \Upsilon}}
\nonumber\\
& \leq \Upsilon(f_0) + \mathbb E_{D_n} h_{f_0} - \mathbb E_{D_n} h_{\wideparen{f}_{D_n, \Upsilon}} 
       + \mathbb E_P h_{\wideparen{f}_{D_n, \Upsilon}} + \delta
\nonumber\\
& = (\Upsilon (f_0) + \mathbb E_P h_{f_0}) 
     + (\mathbb E_{D_n} h_{f_0} - \mathbb E_P h_{f_0}) 
     + (\mathbb E_P h_{\wideparen{f}_{D_n, \Upsilon}} - \mathbb E_{D_n} h_{\wideparen{f}_{D_n, \Upsilon}}) 
     + \delta. 
\label{splitting}
\end{align}

\textbf{Estimating the First Stochastic Term.}
Let us first bound the term $\mathbb E_{D_n} h_{f_0} - \mathbb E_P h_{f_0}$. 
To this end, we further split this difference into
\begin{align}
\mathbb E_{D_n} h_{f_0} - \mathbb E_P h_{f_0} 
= \left( \mathbb E_{D_n} ( h_{f_0} - h_{\wideparen{f}_0} ) - \mathbb E_P ( h_{f_0} - h_{\wideparen{f}_0} ) \right) 
  + ( \mathbb E_{D_n} h_{\wideparen{f}_0} - \mathbb E_P h_{\wideparen{f}_0} ). 
\label{splitting2}
\end{align}
Now $L \circ f_0 - L \circ \wideparen{f}_0 \geq 0$ implies 
$h_{f_0} - h_{\wideparen{f}_0} = L \circ f_0 - L \circ \wideparen{f}_0 \in [0, B_0]$, and hence we obtain
\begin{align*}
\mathbb E_P \left( ( h_{f_0} - h_{\wideparen{f}_0} ) - \mathbb E_P ( h_{f_0} - h_{\wideparen{f}_0} ) \right)^2 
\leq \mathbb E_P ( h_{f_0} - h_{\wideparen{f}_0} )^2 
\leq B_0 \mathbb E_P ( h_{f_0} - h_{\wideparen{f}_0} ).
\end{align*}
Moreover, we find
\begin{align*}
 \|h_{f_0} - h_{\wideparen{f}_0}\| 
& = \|(L \circ f_0 - L \circ f_{L,P}^*) - (L \circ \wideparen{f}_0 - L \circ f_{L,P}^*)\|
\\
& = \|L \circ f_0 - L \circ \wideparen{f}_0\|
  \leq \|L \circ f_0\| + \|L \circ \wideparen{f}_0\|
  \leq 2 A_0.
\end{align*}
Inequality (\ref{bernsteininequality2}) applied to 
$h := ( h_{f_0} - h_{\wideparen{f}_0} ) - \mathbb E_P ( h_{f_0} - h_{\wideparen{f}_0} )$ 
thus shows that for 
\begin{align*}
n \geq n_0^* \geq \max \left\{ \min \left\{ m \geq 3 : 
m^2 \geq \frac{808 c (6 A_0 + B_0)}{B_0} \text{ and } 
\frac{m}{(\log m)^{\frac{2}{\gamma}}} \geq 4 \right\},
e^{\frac{3}{b}} \right\},
\end{align*}
we have
\begin{align*}
\mathbb E_{D_n} ( h_{f_0} - h_{\wideparen{f}_0} ) - \mathbb E_P ( h_{f_0} - h_{\wideparen{f}_0} ) 
\leq \sqrt{\frac{8 (\log n)^{\frac{2}{\gamma}} \tau B_0 \mathbb E_P ( h_{f_0} - h_{\wideparen{f}_0} )}{n}} 
     + \frac{8 (\log n)^{\frac{2}{\gamma}} B_0 \tau}{3 n}
\end{align*}
with probability $\mu$ not less than $1 - 2 e^{-\tau}$.
Moreover, using $\sqrt{a b} \leq \frac{a}{2} + \frac{b}{2}$, we find
\begin{align*}
\sqrt{8 (\log n)^{\frac{2}{\gamma}} n^{-1} \tau B_0 \mathbb E_P ( h_{f_0} - h_{\wideparen{f}_0} )} 
\leq \mathbb E_P ( h_{f_0} - h_{\wideparen{f}_0} ) + 2 (\log n)^{\frac{2}{\gamma}} n^{-1} B_0 \tau,
\end{align*}
and consequently we have with probability $\mu$ not less than $1 - 2 e^{-\tau}$ that
\begin{align}
\mathbb E_{D_n} ( h_{f_0} - h_{\wideparen{f}_0} ) - \mathbb E_P ( h_{f_0} - h_{\wideparen{f}_0} ) 
\leq \mathbb E_P ( h_{f_0} - h_{\wideparen{f}_0} ) 
     + \frac{14 (\log n)^{\frac{2}{\gamma}} B_0 \tau}{3 n}. 
\label{splitting21estimate}
\end{align}
In order to bound the remaining term in (\ref{splitting2}), that is 
$\mathbb E_{D_n} h_{\wideparen{f}_0} - \mathbb E_P h_{\wideparen{f}_0}$, 
we first observe that 
the assumed $L(x, y, t) \leq 1$ for all $(x, y) \in X \times Y$ and $t, t' \in [- M, M]$
implies $\|h_{\wideparen{f}_0}\|_{\infty} \leq 1$, and hence we have 
$\|h_{\wideparen{f}_0} - \mathbb E_P h_{\wideparen{f}_0}\|_{\infty} \leq 2$. 
Furthermore, we have 
\begin{align*}
 \|h_{f_0}\| 
= \|L \circ f_0 - L \circ f_{L,P}^*\|
\leq \|L \circ f_0\| + \|L \circ f_{L,P}^*\|
\leq A_0 + A^*.
\end{align*}
Moreover, (\ref{variancebound}) yields
\begin{align*}
\mathbb E_P ( h_{\wideparen{f}_0} - \mathbb E_P h_{\wideparen{f}_0} )^2 
\leq \mathbb E_P h_{\wideparen{f}_0}^2 
\leq V (\mathbb E_P h_{\wideparen{f}_0})^{\vartheta}.
\end{align*}
In addition, if $\vartheta \in (0, 1]$, the second inequality in Lemma \ref{younginequality} implies for 
$q := \frac{2}{2-\vartheta}$, $q' := \frac{2}{\vartheta}$, 
$a := ((\log n)^{\frac{2}{\gamma}} n^{-1} 2^{3-\vartheta} \vartheta^{\vartheta} V \tau)^{1/2}$, and 
$b := (2 \vartheta^{-1} \mathbb E_P h_{\wideparen{f}_0})^{\vartheta/2}$, that
\begin{align*}
\sqrt{\frac{8 (\log n)^{\frac{2}{\gamma}} V \tau (\mathbb E_P h_{\wideparen{f}_0})^{\vartheta}}{n}} 
& \leq \left( 1 - \frac{\vartheta}{2} \right) 
     \left( \frac{2^{3-\vartheta} \vartheta^{\vartheta} (\log n)^{\frac{2}{\gamma}} V \tau}{n} \right)^{\frac{1}{2 - \vartheta}} 
     + \mathbb E_P h_{\wideparen{f}_0} 
\\
& \leq \left( \frac{8 (\log n)^{\frac{2}{\gamma}} V \tau}{n} \right)^{\frac{1}{2 - \vartheta}} 
       + \mathbb E_P h_{\wideparen{f}_0}.
\end{align*}
Since $\mathbb E_P h_{\wideparen{f}_0} \geq 0$, this inequality also holds for $\vartheta = 0$, 
and hence (\ref{bernsteininequality2}) shows that for 
\begin{align*}
n \geq n_0^* \geq \max \left\{ \min \left\{ m \geq 3 : 
m^2 \geq \frac{808 c (3 A_0 + 3 A^* + 2)}{2} \text{ and } 
\frac{m}{(\log m)^{\frac{2}{\gamma}}} \geq 4 \right\},
e^{\frac{3}{b}} \right\},
\end{align*}
we have
\begin{align}
\mathbb E_{D_n} h_{\wideparen{f}_0} - \mathbb E_P h_{\wideparen{f}_0} 
< \mathbb E_P h_{\wideparen{f}_0} 
  + \left( \frac{8 (\log n)^{\frac{2}{\gamma}} V \tau}{n} \right)^{\frac{1}{2 - \vartheta}} 
  + \frac{16 (\log n)^{\frac{2}{\gamma}} \tau}{3 n} 
\label{splitting22estimate}
\end{align}
with probability $\mu$ not less than $1 - 2 e^{-\tau}$.
By combining this estimate with (\ref{splitting21estimate}) and (\ref{splitting2}), 
we now obtain that with probability $\mu$ not less than $1 - 4 e^{-\tau}$ we have
\begin{align}
\mathbb E_{D_n} h_{f_0} - \mathbb E_P h_{f_0} 
< \mathbb E_P h_{f_0} + \left( \frac{8 (\log n)^{\frac{2}{\gamma}} V \tau}{n} \right)^{\frac{1}{2-\vartheta}} 
  + \frac{16 (\log n)^{\frac{2}{\gamma}} \tau}{3 n} 
  + \frac{14 (\log n)^{\frac{2}{\gamma}} B_0 \tau}{3 n}, 
\label{splitting2estimate}
\end{align}
since $1 \leq B_0$, i.e., we have established a bound on the second term in (\ref{splitting}). 

\textbf{Estimating the Second Stochastic  Term.}
For the third term in (\ref{splitting}) let us first consider the case 
$n / (\log n)^{\frac{2}{\gamma}} < 8 (\tau + \varphi(\varepsilon/2) 2^p r^p)$. 
Combining (\ref{splitting2estimate}) with (\ref{splitting}) and using $1 \leq B_0$, 
$1 \leq V$, and 
$\mathbb E_P h_{\wideparen{f}_{D_n, \Upsilon}} - \mathbb E_{D_n} h_{\wideparen{f}_{D_n, \Upsilon}} \leq 2$, then we find
\begin{align*}
& \Upsilon(f_{D_n, \Upsilon}) + \mathcal R_{L, P} (\wideparen{f}_{D_n, \Upsilon}) - \mathcal R_{L, P}^*
\\
& \leq \Upsilon(f_0) + 2 \mathbb E_P h_{f_0} 
       + \left( \frac{8 (\log n)^{\frac{2}{\gamma}} V \tau}{n} \right)^{\frac{1}{2 - \vartheta}} 
       + \frac{16 (\log n)^{\frac{2}{\gamma}} \tau}{3 n} 
       + \frac{14 (\log n)^{\frac{2}{\gamma}} B_0 \tau}{3 n}
\\
& \phantom{=} + (\mathbb E_P h_{\wideparen{f}_{D_n, \Upsilon}} - \mathbb E_{D_n} h_{\wideparen{f}_{D_n, \Upsilon}} ) + \delta
\\
& \leq \Upsilon(f_0) + 2 \mathbb E_P h_{f_0} 
       + \left( \frac{8 (\log n)^{\frac{2}{\gamma}} V (\tau + \varphi(\varepsilon/2) 2^p r^p)}{n} \right)^{\frac{1}{2 - \vartheta}} 
       + \frac{10 (\log n)^{\frac{2}{\gamma}} B_0 \tau}{n}
\\
& \phantom{=} + 2 \left( \frac{8 (\log n)^{\frac{2}{\gamma}} V (\tau + \varphi(\varepsilon/2) 2^p r^p)}{n} \right)^{\frac{1}{2 - \vartheta}} + \delta
\\
& \leq 2 \Upsilon(f_0) + 4 \mathbb E_P h_{f_0} 
       + 3 \left( \frac{24 (\log n)^{\frac{2}{\gamma}} V (\tau + \varphi(\varepsilon/2) 2^p r^p)}{3 n} \right)^{\frac{1}{2 - \vartheta}} 
       + \frac{10 (\log n)^{\frac{2}{\gamma}} B_0 \tau}{n} + 2 \delta
\\
& \leq 2 \Upsilon(f_0) + 4 \mathbb E_P h_{f_0} + 4 r + 2 \delta
\end{align*}
with probability $\mu$ not less than $1 - 4 e^{-\tau}$. It thus remains to consider the case 
$n / (\log n)^{\frac{2}{\gamma}} \geq 8 (\tau + \varphi(\varepsilon/2) 2^p r^p)$. 

\textbf{Introduction of the Quotients.}
To establish a non-trivial bound on the term $\mathbb E_P h_{\wideparen{f}_D} - \mathbb E_{D_n} h_{\wideparen{f}_D}$ 
in (\ref{splitting}), we define functions
\begin{align*}
g_{f,r} := \frac{\mathbb E_P h_{\wideparen{f}} - h_{\wideparen{f}}}{\Upsilon(f) + \mathbb E_P h_{\wideparen{f}} + r}, 
~~~~~~ f \in \mathcal F, ~ r > r^*.
\end{align*}
For $f \in \mathcal F_r$, we have $\| \mathbb E_P h_{\wideparen{f}} - h_{\wideparen{f}} \|_{\infty} \leq 2$.
Furthermore, for $f \in \mathcal F_r$ and $k \geq 0$ with $2^k r \leq 1$, 
by the assumption (\ref{lipfr}) we find 
\begin{align*}
 \|h_{\wideparen{f}}\| 
= \|L \circ \wideparen{f} - L \circ f_{L,P}^*\|
\leq \|L \circ \wideparen{f}\| + \|L \circ f_{L,P}^*\|
\leq A_{2^k r} + A^*
\leq A_1 + A^*.
\end{align*}
Moreover, for $f \in \mathcal F_r$, the variance bound (\ref{variancebound}) implies
\begin{align}
\mathbb E_P ( h_{\wideparen{f}} - \mathbb E_P h_{\wideparen{f}} )^2 
\leq \mathbb E_P h_{\wideparen{f}}^2 
\leq V (\mathbb E_P h_{\wideparen{f}})^{\vartheta} 
\leq V r^{\vartheta}. \label{variancebound2}
\end{align}

\textbf{Peeling.}
This part is completely identical to the part \textit{Peeling} on page 135 of our work \cite{HaSt14a}. 
Hence we have neglected some steps of the derivations.
In case of uncertainty one may refer to \cite{HaSt14a} for details.

For a fixed $r \in (r^*, 1]$, let $K$ be the largest integer satisfying $2^K r \leq 1$. 
Then we can get the following disjoint partition of the function set $\mathcal F_1$:
\begin{align}
\mathcal F_1 \subset \mathcal F_r \cup \bigcup_{k=1}^{K+1} \left( \mathcal F_{2^k r} \backslash \mathcal F_{2^{k-1} r} \right).
\label{construction}
\end{align}
We further write $\overline{C}_{\varepsilon, r, 0}$ for a minimal $\varepsilon$-net of $\mathcal F_r$ and 
$\overline{C}_{\varepsilon, r, k}$ for minimal $\varepsilon$-nets of $\mathcal F_{2^k r} \backslash \mathcal F_{2^{k-1} r}$, 
$1 \leq k \leq K+1$, respectively. Then the union of these nets 
$\overline{C}_{\varepsilon, 1} := \bigcup_{k=0}^{K+1} \overline{C}_{\varepsilon, r, k}$ 
is an $\varepsilon$-net of the set $\mathcal F_1$. Moreover, we define 
\begin{align}
\widetilde{\mathcal C}_{\varepsilon, r, k} 
:= \bigcup_{l=0}^k \overline{C}_{\varepsilon, r, l}, 
~~~~ 0 \leq k \leq K + 1.
\label{overlapp}
\end{align}
Then we have $\overline{C}_{\varepsilon, 1} = \bigcup_{k=0}^{K+1} \widetilde{C}_{\varepsilon, r, k}$. 
Moreover, the cardinality of $\widetilde{\mathcal C}_{\varepsilon, r, k}$ can be estimated by
\begin{align}
 |\widetilde{\mathcal C}_{\varepsilon, r, k}| 
\leq (k+1) \exp \left( \varphi(\varepsilon/2) 2^{k p} r^p \right), ~~~~ 0 \leq k \leq K + 1.
\label{modicoveringnumber}
\end{align}
Then, peeling by \cite[Theorem 5.2]{HaSt14a} implies 
\begin{align}
\mu \left( \sup_{f \in \overline{\mathcal C}_{\varepsilon, 1}} \mathbb E_{D_n} g_{f,r} > \frac{1}{4} \right)
\leq 2 \sum_{k=1}^{K+1} \mu \left( \sup_{f \in \widetilde{\mathcal C}_{\varepsilon, r, k}} 
                                      \mathbb E_{D_n} (\mathbb E_P h_{\wideparen{f}} - h_{\wideparen{f}}) > 2^{k-3} r \right).
\label{estimationbypeeling}
\end{align}

\textbf{Estimating the Error Probabilities on the ``Spheres''.}
Our next goal is to estimate all the error probabilities on the right-hand side of (\ref{estimationbypeeling}).
By our construction, we have $\widetilde{\mathcal C}_{\varepsilon, r, k} \subset \mathcal F_{2^k r}$.
This, together with (\ref{bernsteininequality1}), (\ref{variancebound2}), the union bound 
and the estimates of the covering numbers (\ref{modicoveringnumber}), implies that for 
\begin{align*}
n \geq n_0^* \geq \max \left\{ \min \left\{ m \geq 3 : 
m^2 \geq \frac{808 c (3 A_1 + 3 A^* + 2)}{2} \text{ and } 
\frac{m}{(\log m)^{\frac{2}{\gamma}}} \geq 4 \right\},
e^{\frac{3}{b}} \right\},
\end{align*}
we have
\begin{align}
& \mu \left( \sup_{f \in \widetilde{\mathcal C}_{\varepsilon, r, k}} 
              \mathbb E_{D_n} (\mathbb E_P h_{\wideparen{f}} - h_{\wideparen{f}}) > 2^{k-3} r \right)
\nonumber\\
& \leq 2 |\widetilde{\mathcal C}_{\varepsilon, r, k}| 
         \exp \left( - \frac{n}{8 (\log n)^{\frac{2}{\gamma}}} \cdot 
                       \frac{(2^{k-3} r)^2}{V (2^k r)^{\vartheta} + 2 (2^{k-3} r)/3} \right)
\nonumber\\
& \leq 2 (k+1) \exp \left( \varphi(\varepsilon/2) 2^{k p} r^p \right) 
               \cdot \exp \left( - \frac{n}{8 (\log n)^{\frac{2}{\gamma}}} \cdot 
                                   \frac{3 (2^{k-1} r)^2}{96 V (2^{k-1} r)^{\vartheta} + 8 (2^{k-1} r)} \right),
\label{smallprobability}
\end{align}
since $\vartheta \in [0, 1]$.
For $k \geq 1$, we denote the right-hand side of this estimate by $p_k(r)$, that is
\begin{align}
p_k(r) := 2 (k+1) \exp \left( \varphi(\varepsilon/2) 2^{k p} r^p \right) 
                  \cdot \exp \left( - \frac{n}{8 (\log n)^{\frac{2}{\gamma}}} \cdot 
                                      \frac{3 (2^{k-1} r)^2}{96 V (2^{k-1} r)^{\vartheta} + 8 (2^{k-1} r)} \right).
\label{pkr}
\end{align}
Then, as derived in \cite{HaSt14a}, we can obtain
\begin{align*}
q_k(r) := \frac{p_{k+1}(r)}{p_k(r)}
\leq 2 \exp \left( \varphi(\varepsilon/2) 2^{k p + 1} r^p \right) 
             \cdot \exp \left( - \frac{n}{8 (\log n)^{\frac{2}{\gamma}}} \cdot 
                                 \frac{3 (2^{k-1} r)^2}{96 V (2^{k-1} r)^{\vartheta} + 8 (2^{k-1} r)} \right),
\end{align*}
and our assumption $2^k r \leq 1$, $0 \leq k \leq K$ implies
\begin{align*}
q_k(r) 
& \leq 2 \exp \left( \varphi(\varepsilon/2) 2^{kp+1} r^p \right)
        \cdot \exp \left( - \frac{n}{8 (\log n)^{\frac{2}{\gamma}}} \cdot 
                            \frac{3 (2^{k-1} r)^2}{96 V (2^{k-1} r)^{\vartheta} + 8 (2^{k-1} r)^{\vartheta}} \right)
\\
& \leq 2 \exp \left( 2^{(k-1)p} \cdot 4 r^p \varphi(\varepsilon/2) 
         - \frac{2^{(k - 1) (2 - \vartheta)} \cdot 3 n r^{2 - \vartheta}}{64 (12 V + 1) (\log n)^{\frac{2}{\gamma}}}  \right).
\end{align*}
Since $p \in (0, 1]$, $k \geq 1$ and $\vartheta \in [0, 1]$, we have
$2^{(k-1)p} \leq 2^{(k - 1) (2 - \vartheta)}$.
Then the first assumption in (\ref{minradius}), namely,
\begin{align*}
r \geq \left( \frac{512 (12 V + 1) (\log n)^{\frac{2}{\gamma}} (\tau + \varphi(\varepsilon/2) 2^p r^p)}{3 n} \right)^{\frac{1}{2 - \vartheta}}
\end{align*}
implies that
$3 n r^{2 - \vartheta} \geq 512 (12 V + 1) (\log n)^{\frac{2}{\gamma}} \varphi(\varepsilon/2) r^p$.
By using $2^{(k - 1) (2 - \vartheta)} \geq 1$, we find
\begin{align*}
q_k(r) \leq 2 \exp \left( - \frac{3 n r^{2 - \vartheta}}{128 (12 V + 1) (\log n)^{\frac{2}{\gamma}}} \right). 
\end{align*}
Moreover, since $\tau \geq 1$, the first assumption in (\ref{minradius}) implies also
$3 n r^{2 - \vartheta} \geq 4 \cdot 128 (12 V + 1) (\log n)^{\frac{2}{\gamma}}$.
Hence we have $q_k(r) \leq 2 e^{-4}$, that is, 
\begin{align}
 p_{k+1}(r) \leq 2 e^{-4} p_k(r)
\,\,\,\,\,\,\,\, 
\text{ for all } k \geq 1.
\label{relation} 
\end{align}

\textbf{Summing all the Error Probabilities.}
Now, combining (\ref{estimationbypeeling}) with (\ref{smallprobability}),
(\ref{pkr}), and (\ref{relation}), we obtain
\begin{align*}
\mu \left( \sup_{f \in \overline{\mathcal C}_{\varepsilon, 1}} \mathbb E_{D_n} g_{f,r} > \frac{1}{4} \right) 
& \leq 2 \sum_{k=1}^{K+1} p_k(r) 
  \leq 3 p_1(r)
\\
& = 12 \exp \left( \varphi(\varepsilon/2) 2^p r^p \right) 
         \cdot \exp \left( - \frac{n}{8 (\log n)^{\frac{2}{\gamma}}} \cdot 
                             \frac{3 r^2}{96 V r^{\vartheta} + 8 r} \right)
\\
& \leq 12 \exp \left( \varphi(\varepsilon/2) 2^p r^p \right) 
          \cdot \exp \left( - \frac{3 n r^{2 - \vartheta}}{64 (12 V + 1) (\log n)^{\frac{2}{\gamma}}}  \right),
\end{align*}
where in the last step we used $r \in (0, 1]$ and $\vartheta \in [0, 1]$.
Then once again the first assumption in (\ref{minradius}) gives
$3 n r^{2 - \vartheta} \geq 64 (12 V + 1) (\log n)^{\frac{2}{\gamma}} (\tau + \varphi(\varepsilon/2) 2^p r^p)$
and a simple transformation thus yields
\begin{align*}
\mu \left( D_n \in (X \times Y)^n : \sup_{f \in \overline{\mathcal C}_{\varepsilon, 1}} \mathbb E_{D_n} g_{f,r} \leq \frac{1}{4} \right) 
\geq 1 - 12 e^{-\tau}.
\end{align*} 
The rest of the argument is completely analogous to the proof of \cite[Theorem 3.1]{HaSt14a}
and the assertion is proved.
\end{proof}

\begin{proof}[Proof of Theorem \ref{LSSVMgaussiankernels}]

For the least-square loss, the variance bound (\ref{variancebound}) is valid with $\vartheta = 1$, hence
the condition (\ref{minradius}) is satisfied if
\begin{align}
r \geq \max \Bigg\{ 
\left( c_V 2^{1+3p} a \right)^{\frac{1}{1-p}} 
\sigma^{-\frac{d}{1-p}} \lambda^{-\frac{p}{1-p}} 
\left( \frac{n}{(\log n)^{\frac{2}{\gamma}}} \right)^{\frac{1}{1-p}} 
\varepsilon^{-\frac{2p}{1-p}},
\frac{2 c_V (\log n)^{\frac{2}{\gamma}} \tau}{n},
\frac{20 B_0 (\log n)^{\frac{2}{\gamma}} \tau}{n},
r^* \Bigg\}.
\label{minradius3}
\end{align}
Furthermore,  \cite[Section 2]{EbSt11a}
 shows that there exists a constant $c>0$ such that 
 for all  $\sigma \in (0, 1]$, there is 
 an $f_0\in H_\sigma$ with $\inorm{f_0}\leq c$, 
$\snorm{f_0}_{H_\sigma}^2 \leq c \sigma^{-d}$,
and 
\begin{displaymath}
\mathcal R_{L, P}(f_0) -  \mathcal R_{L, P}^*  \leq  c \sigma^{2t} \, .
\end{displaymath}
Moreover, \cite[Lemma 5.5]{StAn09a} shows 
every function $f$ in $H_{\sigma}$ is
Lipschitz continuous with 
\begin{align*}
|f|_1 \leq \sqrt{2} \sigma^{-1} \|f\|_{H_{\sigma}(X)}\, , 
\end{align*}
and this implies
\begin{align*}
 |\wideparen{f}_0|_1
\leq |f_0|_1
\leq \sqrt{2}  \sigma^{-1} \|f_0\|_{H_{\sigma}(X)}
\leq \sqrt{2} c\sigma^{-1} .
\end{align*}
Moreover, there exists a constant $C^* < \infty$ such that $|f_{L,P}^*|_1 \leq C^*$,
since we have assumed that $f^*_{L, P} \in \mathrm{Lip}(\mathbb R^d)$. Then,
Lemma \ref{lipschitzlemma} (i) yields 
\begin{align*}
 4 A_0 + A_1 + A^* + 1
& = 
2\sqrt{2} \left( M + \|f\|_{\infty} \right) 
\left(
4 + 4 |f_0|_1
+ 1 + \sup_{f \in \mathcal{F}_1} |f|_1
+ 1 + |f_{L,P}^*|_1
+ 1
\right) + 1
\\
& \leq
2\sqrt{2} \left( M + \|f\|_{\infty} \right) 
\left( 7 +
4 \sqrt{2} c \sigma^{-1} 
+ \sup_{r \leq 1} \sqrt{2} \sigma^{-1} \lambda^{-1/2} r^{1/2}
+ C^*
\right) + 1
\\
& =
2\sqrt{2} \left( M + \|f\|_{\infty} \right) 
\left( 7 +
5 \sqrt{2} c \sigma^{-1} \lambda^{-1/2}
+ C^*
\right) + 1
\\
& \leq 2 C \sigma^{-1} \lambda^{-1/2} \leq 2 C n
\end{align*}
for all $\sigma, \lambda \in (0, 1]$ with $\lambda \sigma^2 \geq n^{-2}$,
where  $C$ is a constant independent of $n$, $\lambda$, and $\sigma$. 
For 
\begin{align*}
n \geq \max \left\{ 2 C, \min \left\{ m \geq 3 :
\frac{m}{(\log m)^{\frac{2}{\gamma}}} \geq 4 \right\},
e^{\frac{3}{b}} \right\},
\end{align*}
the oracle inequality (\ref{oracleinequalityy}) thus implies
\begin{align*}
& \lb\snorm{f_{D_n, \lb}}_{H_\sigma}^2 + \mathcal R_{L, P} (\wideparen{f}_{D_n, \lb}) - \mathcal R_{L, P}^*
\nonumber\\
& \leq 4 \lb\snorm{f_0}_{H_\sigma}^2 + 4 \mathcal R_{L, P}(f_0) - 4 \mathcal R_{L, P}^* + 4 r + 5 \varepsilon  
\nonumber\\
& \leq C_1 \left( \lambda \sigma^{-d} +  \sigma^{2t}
+ \sigma^{-\frac{d}{1-p}} \lambda^{-\frac{p}{1-p}} 
\left( \frac{n}{(\log n)^{\frac{2}{\gamma}}} \right)^{\frac{1}{1-p}}  \varepsilon^{-\frac{2p}{1-p}} \t
+ \varepsilon \right),
\end{align*}
where  $C_1$ is a constant independent of $n$, $\lb$, $\sigma$, $\t$, and $\e$. Here, 
$\lb$, $\s$, and $n$ need satisfy the additional requirement $\sigma, \lambda \in (0, 1]$ with $\lambda \sigma^2 \geq n^{-2}$.
Now, optimizing over $\e$ by using   \cite[Lemma A.1.5]{StCh08a}, we get
\begin{align}
\lb\snorm{f_{D_n, \lb}}_{H_\sigma}^2 + \mathcal R_{L, P} (\wideparen{f}_{D_n, \lb}) - \mathcal R_{L, P}^*
\leq C_2 \left( \lambda \sigma^{-d} +  \sigma^{2t}
+ \sigma^{-\frac{d}{1+p}} \lambda^{-\frac{p}{1+p}} \left( \frac{n}{(\log n)^{\frac{2}{\gamma}}} \right)^{\frac{1}{1+p}} \t \right),
\label{zwischenschritt}
\end{align}
where  $C_2$ is a constant independent of $n$, $\lb$, and $\sigma$. By applying
\cite[Lemma A.1.6]{StCh08a}, we can optimize the right-hand side of (\ref{zwischenschritt}) 
over $\lambda$ and $\sigma$, then we see that for all
$\xi > 0$ we can find $p, \zeta \in (0, 1)$ sufficiently close to $0$ such that 
the LS-SVM using Gaussian RKHS $H_{\sigma}$ and
$\lambda_n = n^{- 1}$, $\sigma_n = n^{- \frac{1}{2 t + d}}$  
learns with rate 
$n^{- \frac{2 t}{2 t + d} + \xi}$,
since the requirement $\lambda_n \sigma_n^2 \geq n^{-2}$ is automatically satisfied by the assumed $t\geq 1$.
\end{proof}

\begin{proof}[Proof of Theorem \ref{oracleinequalityds}]
Theorem \ref{oracleinequality} yields
\begin{align*}
\Upsilon(f_{\!\boldsymbol{D}_{\!n}^{(j)}, \Upsilon}) 
+ \mathbb E_{P} h_{\wideparen{f}_{\!\boldsymbol{D}_{\!n}^{(j)}, \Upsilon}} 
\leq 2 \Upsilon(f_0) + 4 \mathbb E_{P} h_{f_0} + 4 r + 5 \varepsilon + 2 \delta
\end{align*}
with probability $\mu \otimes \nu$ not less than $1 - 16 e^{-\tau}$. 
Using (\ref{dto1}) and the definition (\ref{forecasterd}) we
then easily obtain the assertion.
\end{proof}

{\small
\bibliographystyle{abbrv}

}

\end{document}